\documentclass[10pt, oneside]{amsart}
\usepackage{amsfonts, amstext, amsmath, amsthm, amscd, amssymb}
\usepackage{manfnt}
\usepackage{tikz}
\usepackage{graphicx}
\usepackage{url}

\usepackage{etex}

\usepackage[
  textwidth=42mm,
  backgroundcolor=yellow,
  bordercolor=orange,
  textsize=small
]{todonotes}

\usetikzlibrary{calc,decorations.pathreplacing}

\usepackage{marginnote}

\usepackage[paper=a4paper, text={137mm,208mm}, centering]{geometry}

\usepackage{graphics, color}
\usepackage{comment}

\usepackage[all,graph]{xy}
\usepackage[bookmarks]{hyperref}
\usepackage{float, placeins, booktabs, longtable}

\usepackage{relsize}

\iftrue
\makeatletter
\def\@settitle{%
  \vspace*{-20pt}
  \begin{flushleft}%
    \baselineskip14\p@\relax
    \normalfont\bfseries\LARGE
    \@title
  \end{flushleft}%
}
\def\@setauthors{%
  \begingroup
  \def\thanks{\protect\thanks@warning}%
  \trivlist
  \large \@topsep30\p@\relax
  \advance\@topsep by -\baselineskip
  \item\relax
  \author@andify\authors
  \def\\{\protect\linebreak}%
  \authors
  \ifx\@empty\contribs
  \else
    ,\penalty-3 \space \@setcontribs
    \@closetoccontribs
  \fi
  \normalfont
  \endtrivlist
  \endgroup
}
\def\@setaddresses{\par
  \nobreak \begingroup\raggedright
  \small
  \def\author##1{\nobreak\addvspace\smallskipamount}%
  \def\\{\unskip, \ignorespaces}%
  \interlinepenalty\@M
  \def\address##1##2{\begingroup
    \par\addvspace\bigskipamount\noindent
    \@ifnotempty{##1}{(\ignorespaces##1\unskip) }%
    {\ignorespaces##2}\par\endgroup}%
  \def\curraddr##1##2{\begingroup
    \@ifnotempty{##2}{\nobreak\noindent\curraddrname
      \@ifnotempty{##1}{, \ignorespaces##1\unskip}\/:\space
      ##2\par}\endgroup}%
  \def\email##1##2{\begingroup
    \@ifnotempty{##2}{\smallskip\nobreak\noindent E-mail address%
      \@ifnotempty{##1}{, \ignorespaces##1\unskip}\/:\space
      \ttfamily##2\par}\endgroup}%
  \def\urladdr##1##2{\begingroup
    \def~{\char`\~}%
    \@ifnotempty{##2}{\nobreak\noindent\urladdrname
      \@ifnotempty{##1}{, \ignorespaces##1\unskip}\/:\space
      \ttfamily##2\par}\endgroup}%
  \addresses
  \endgroup
  \global\let\addresses=\@empty
}
\def\@setabstracta{%
    \ifvoid\abstractbox
  \else
    \skip@25\p@ \advance\skip@-\lastskip
    \advance\skip@-\baselineskip \vskip\skip@
    \box\abstractbox
    \prevdepth\z@ 
    \vskip-10pt
  \fi
}
\renewenvironment{abstract}{%
  \ifx\maketitle\relax
    \ClassWarning{\@classname}{Abstract should precede
      \protect\maketitle\space in AMS document classes; reported}%
  \fi
  \global\setbox\abstractbox=\vtop \bgroup
    \normalfont\small
    \list{}{\labelwidth\z@
      \leftmargin0pc \rightmargin\leftmargin
      \listparindent\normalparindent \itemindent\z@
      \parsep\z@ \@plus\p@
      
    }%
    \item[\hskip\labelsep\bfseries\abstractname.]%
}{%
  \endlist\egroup
  \ifx\@setabstract\relax \@setabstracta \fi
}
\def\section{\@startsection{section}{1}%
  \z@{-1.2\linespacing\@plus-.5\linespacing}{.8\linespacing}%
  {\normalfont\bfseries\large}}
\def\subsection{\@startsection{subsection}{2}%
  \z@{-.8\linespacing\@plus-.3\linespacing}{.3\linespacing\@plus.2\linespacing}%
  {\normalfont\bfseries}}
\def\subsubsection{\@startsection{subsubsection}{3}%
  \z@{.7\linespacing\@plus.1\linespacing}{-1.5ex}%
  {\normalfont\itshape}}
\def\@secnumfont{\bfseries}
\makeatother
\fi 

\def\to{\mathchoice{\longrightarrow}{\rightarrow}{\rightarrow}{\rightarrow}}
\makeatletter
\newcommand{\shortxra}[2][]{\ext@arrow 0359\rightarrowfill@{#1}{#2}}
\def\longrightarrowfill@{\arrowfill@\relbar\relbar\longrightarrow}
\newcommand{\longxra}[2][]{\ext@arrow 0359\longrightarrowfill@{#1}{#2}}
\renewcommand{\xrightarrow}[2][]{\mathchoice{\longxra[#1]{#2}}%
  {\shortxra[#1]{#2}}{\shortxra[#1]{#2}}{\shortxra[#1]{#2}}}
\makeatother

\makeatletter
\def\Nopagebreak{\@nobreaktrue\nopagebreak}
\makeatother

\makeatletter
\newtheoremstyle{theorem-giventitle}
        {}{}              
        {\itshape}                      
        {}                              
        {\bfseries}                     
        {.}                             
        {\thm@headsep}                             
        {\thmnote{\bfseries#3}}
\newtheoremstyle{theorem-givenlabel}
        {}{}              
        {\itshape}                      
        {}                              
        {\bfseries}                     
        {.}                             
        {\thm@headsep}                             
        {\thmname{#1}~\thmnumber{#3}\setcurrentlabel{#3}}

\newtheoremstyle{definition-giventitle}
        {}{}              
        {}                      
        {}                              
        {\bfseries}                     
        {.}                             
        {\thm@headsep}                             
        {\thmnote{\bfseries#3}}
\def\setcurrentlabel#1{\gdef\@currentlabel{#1}}
\makeatother

\newcommand{\bp}{\begin{pmatrix}}
\newcommand{\ep}{\end{pmatrix}}
\newcommand{\be}{\begin{equation}}
\newcommand{\ee}{\end{equation}}
\newcommand{\ol}[1]{\overline{#1}}

\def\sbmatrix#1{\left[\begin{smallmatrix}#1\end{smallmatrix}\right]}

\numberwithin{equation}{section}

\theoremstyle{plain}
\newtheorem{theorem}[equation]{Theorem}
\newtheorem{lemma}[equation]{Lemma}
\newtheorem{proposition}[equation]{Proposition}

\newtheorem*{claim*}{Claim}

\theoremstyle{definition}

\newtheorem{remark}[equation]{Remark}

\newtheorem{definition}[equation]{Definition}

\numberwithin{equation}{section}

\theoremstyle{theorem-giventitle}
\newtheorem{theorem-named}{}

\def\Z{\mathbb Z}
\def\R{\mathbb R}
\def\Q{\mathbb Q}

\def\p{\partial}

\def\sm{\setminus}

\def\a{\alpha}

\def\bp{\begin{pmatrix}}
\def\ep{\end{pmatrix}}
\def\ba{\begin{array}}
\def\ea{\end{array}}
\def\bn{\begin{enumerate}}
\def\en{\end{enumerate}}
\def\mbf{\mathbf}

\DeclareMathOperator\Arf{Arf}

\DeclareMathOperator\cl{cl}
\DeclareMathOperator\Wh{Wh}
\DeclareMathOperator\Ext{Ext}
\DeclareMathOperator\Tor{Tor}
\DeclareMathOperator\Hom{Hom}

\DeclareMathOperator\ord{ord}

\def\Im{\operatorname{Im}}

\DeclareMathOperator\Bl{B\ell}
\DeclareMathOperator\inte{int}

\newcommand{\eps}{\varepsilon}

\begin{document}

\title{Whitney towers and abelian invariants of knots}

\author{Jae Choon Cha}
\address{
  Department of Mathematics\\
  POSTECH\\
  Pohang Gyeongbuk 37673\\
  Republic of Korea
  \linebreak
  School of Mathematics\\
  Korea Institute for Advanced Study \\
  Seoul 02455\\
  Republic of Korea
}
\email{jccha@postech.ac.kr}
\thanks{JCC was partially supported
by NRF grants 2013067043 and 2013053914.}

\author{Kent E. Orr}
\address{
  Department of Mathematics\\
  Indiana University \\
  Bloomington, IN 47405\\
  USA}
\email{korr@indiana.edu}
\thanks{KEO is supported by Simons Foundation Grant \# 430351}

\author{Mark Powell}
\address{
  Department of Mathematics\\ 
  Durham University\\
  United Kingdom
}
\email{mark.a.powell@durham.ac.uk}
\thanks{MP was supported by an NSERC Discovery grant.}

\def\subjclassname{\textup{2010} Mathematics Subject Classification}
\expandafter\let\csname subjclassname@1991\endcsname=\subjclassname
\expandafter\let\csname subjclassname@2000\endcsname=\subjclassname

\begin{abstract}
  We relate certain abelian invariants of a knot, namely the Alexander
  polynomial, the Blanchfield form, and the Arf invariant, to
  intersection data of a Whitney tower in the 4-ball bounded by the
  knot.  We also give a new 3-dimensional algorithm for computing
  these invariants.
\end{abstract}

\maketitle

\section{Introduction}

We show that intersection data  in \emph{Whitney towers} determines
\emph{abelian invariants} of knots, particularly the Blanchfield form,
the Alexander polynomial, and the Arf invariant.

Briefly speaking, a \emph{Whitney tower} traces an iterated attempt to
alter an immersed disc in a 4-manifold to an embedded disc by Whitney
moves. Whitney towers naturally approximate an embedded disc.  In
particular, since the work of
Cochran-Orr-Teichner~\cite{Cochran-Orr-Teichner:1999-1}, Whitney
towers in 4-space have been commonly used to measure the degree to
which a knot fails to be slice.

Our main result algorithmically computes the Blanchfield form and the
Alexander polynomial of a knot using intersection data from an order
two twisted Whitney tower in the 4-disc bounded by the knot.
This relates two incarnations of the \emph{Arf invariant of a knot}
using a $4$-dimensional argument---one characterizing the Arf
invariant in terms of Whitney towers, the other in terms of the
Alexander polynomial.

\subsection{Intersection data from order two towers and abelian
  invariants}

The Seifert pairing provides a well-known method to compute a
presentation for the Alexander module of a knot~\cite{Seifert-1935}.
As a bonus, one easily computes the Alexander polynomial and the Arf
invariant.  The Seifert pairing also gives rise to
a formula for the Blanchfield form of the knot~\cite{Kearton:1975-1,
  Levine-77-knot-modules}.  See also~\cite{Friedl-Powell:2016-1}.

This paper takes a different approach, replacing the Seifert surface
with a Whitney tower in the $4$-disc. This approach promises many
advantages, among these that higher order Whitney towers may present
modules corresponding to nilpotent and solvable covers of the knot.

Every knot $K \subset S^3$ bounds an order two Whitney tower in $D^4$, as we
demonstrate in Section~\ref{section:presentation tower}.  Recall that
this means $K$ is the boundary of an immersed (order 0) disc, $D_0 \looparrowright D^4$,
with $d=2k$ self-intersections occurring in oppositely signed pairs.
Immersed Whitney discs (of order 1), $D^1_1 \cup \cdots \cup D^k_1$,
arise from each of $k$ cancelling pairs of intersection points.
Furthermore, (order two) discs pair order one intersections, which are
intersections between order 0 and order 1 discs.  That is, an order
two Whitney tower is built from immersed Whitney discs which pair all
intersections of order less than $2$ in the tower.

In a neighbourhood of each intersection point, two local discs, called
{\em sheets}, intersect transversely.  We will see that an order 2
Whitney tower can be improved as follows:

\begin{itemize}
\item[(i)] $\pi_1(D^4 \sm \nu D_0) \cong \mathbb Z$;
\item[(ii)] $D_0 \cap \inte D_1^j = \emptyset$ for each~$j$;
  that is, the tower has no order 1 intersections and thus has no
  order 2 discs.
\item[(iii)] For each disc, $D_1^j$, we can choose one of the two
  associated double points.  This double point comes with an immersed
  disc $A^j_1$ in $D^4 \sm \nu D_0$ bounded by a loop leaving the
  double point along one sheet of the intersection and returning to
  the double point along the other sheet; $A^j_i$ is called an {\em
    accessory} disc.
\end{itemize}
We remark that we do not impose any framing conditions on the $D_1^j$
nor on the $A_1^j$.  Experts will know how to construct such a Whitney
tower, but we include a complete proof in
Section~\ref{section:presentation tower}.

\begin{definition}
  \label{definition:pres-tower}
  An order two Whitney tower equipped with accessory discs, namely
  $D_0\cup \big(\bigcup_j D_1^j\big) \cup \big(\bigcup_j A_1^j\big)$,
  is an
    {\em order two presentation tower for~$K$} if the conditions above are satisfied.
\end{definition}

We will view such a tower as a geometric analogue of a presentation
matrix for the Alexander module, one which packages the abelian
invariants we study.

Now we describe such a presentation matrix, arising from the intersection data of the discs in a presentation tower.
Define $W:= D^4 \sm \nu D_0$ to be the exterior of the order zero disc.
The intersection pairing of transverse $2$-chains in $W$ takes values
in the group ring $\Z[\pi_1(W)] = \Z[\Z] \cong \Z[t,t^{-1}]$.
Let $e_{2i-1}=D_1^i$ and $e_{2i}=A_1^i$.  Let $\Lambda=(\lambda_{ij})$
be the $d\times d$ matrix over $\Z[t,t^{-1}]$ whose $(i,j)$-entry,
$\lambda_{ij}$, is the $\Z[t,t^{-1}]$-valued intersection of $e_i$
and~$e_j$. To define the diagonal entry $\lambda_{ii}$, which is the
intersection of $e_i$ and a push-off of $e_i$, we need a section of the
normal bundle of the (Whitney or accessory) disc, along which the push-off
 is taken.  For this purpose we use an extension of the
\emph{Whitney framing} and \emph{accessory framing} of the boundary of
the disc. A detailed description is given in
Sections~\ref{section:intersection-form}
and~\ref{section:framings-definitions}.  For now we remark that the
\emph{twisting} information of the order one Whitney discs and
accessory discs is reflected in these diagonal matrix entries.

Now let $E=(\varepsilon_{ij})$ be the $d\times d$ matrix given
by
\[
  \varepsilon_{ij} =
  \begin{cases}
    \text{the sign of $p$} & \text{if $i=j$ and $e_i$ is an accessory
      disc based at a double point $p$,}
    \\
    1 & \text{if one of $e_i$ and $e_j$ is an accessory disc for a
      double point $p$}
    \\
    & \text{\quad and the other is a Whitney disc with $p$ on the boundary,}
    \\
    0 & \text{otherwise.}
  \end{cases}
\]
 Define $\Omega := z\Lambda+E$
where $z:=(1-t)(1-t^{-1})$.
We say that two polynomials in $\Z[t,t^{-1}]$ are equal \emph{up to norms and units} if they agree in the quotient of $\Z[t,t^{-1}]$ by the multiplicative subgroup
\[\{\pm t^k f(t)f(t^{-1})\,|\, k \in \Z, f(t) \in \Z[t,t^{-1}], |f(1)|=1\}.\]

In the following theorem we show that the matrix $\Omega$ presents the Blanchfield
pairing (see Definition~\ref{defn:presentation-of-blanchfield-form})
up to Witt equivalence, and thus determines the Alexander polynomial
up to norms and units.

\begin{theorem}
  \label{theorem:compute-alex-poly-from-whitney-data}
  The matrix $\Omega$ is a presentation matrix for a linking form Witt
  equivalent to the Blanchfield form of~$K$.  The determinant of
  $\Omega$ equals the Alexander polynomial of $K$, $\Delta_K(t)$, up
  to norms and units.
\end{theorem}

A variation on the above theorem arises by replacing the Whitney discs
with additional accessory discs in the following way.  Recall that the
Whitney disc $e_{2i-1} = D_1^i$ joins two self-intersection points of
$D_0$, say $p_i$ and $q_i$, and the corresponding accessory disc
$e_{2i}=A_1^i$ is based at one of these, say~$p_i$.  Let $e_{2i-1}$ be
an accessory disc for the other intersection point~$q_i$.  Replace an
arbitrary sub-collection of the Whitney discs by accessory discs as
above.  We obtain an intersection matrix via the same prescription
given above and the conclusions of
Theorem~\ref{theorem:compute-alex-poly-from-whitney-data} still hold.
In particular,
Theorem~\ref{theorem:compute-alex-poly-from-whitney-data} holds even
when all the order one discs are accessory discs.  More precisely, let
$e_i$ be an accessory disc for the $i$th double point of $D_0$,
$i=1,\ldots,d$. Let $\Lambda=(\lambda_{ij})$ where $\Lambda_{ij}$ is
the $\Z[\Z]$-intersection number of $e_i$ and $e_j$.  Let $E$ be the
$d \times d$ diagonal matrix whose $i$th diagonal entry is the sign of
the $i$th double point.  Define $\Psi = z\Lambda+E$.

\begin{theorem}
  \label{theorem:main-presentation-Blanchfield}
  The matrix $\Psi$ is a presentation matrix for a linking form Witt
  equivalent to the Blanchfield form of~$K$.  The determinant of
  $\Psi$ equals $\Delta_K(t)$, the Alexander polynomial of $K$, up to norms and units.
\end{theorem}

\subsubsection*{An algorithm to compute abelian invariants}

Here is a special case of the type of tower used to determine the
matrix $\Psi$ in Theorem~\ref{theorem:main-presentation-Blanchfield}.
Construct an immersed disc bounded by a knot $K$ as follows.  Start
with a collection of crossings on a planar diagram of~$K$ such that
changing these crossings gives the trivial knot.  The associated
homotopy traces out a level preserving immersed annulus in
$S^1 \times I \hookrightarrow S^3\times I$ bounded by
$K=K\times \{0\}\subset S^3\times \{0\}$, and a trivial knot in
$S^3\times\{1\}$, whose intersections correspond to the crossing
changes.  Cap off $S^3 \times I$ by gluing a copy of $D^4$ to
$S^3 \times \{1\}$ and cap off the annulus along its unknotted
boundary component to obtain an immersed disc $D_0$ in $D^4$,
which~$K$ bounds.  Choose an accessory disc for each self-intersection
of~$D_0$.  Define $\Psi$ as in
Theorem~\ref{theorem:main-presentation-Blanchfield}.

The next result enables us, in this special case, to compute abelian
invariants from the intersection data of the immersed tower without
the indeterminacy from Witt equivalence and norms.

\begin{theorem}
  \label{theorem:main-blanchfield-alexander}
  The matrix $\Psi$ is a presentation matrix for the Blanchfield form
  of~$K$.  In particular, the determinant of $\Psi$ equals
  $\Delta_K(t)$ up multiplication by a unit $\pm t^k$.
\end{theorem}

In addition, we show that for a special choice of accessory discs, the
computation of the intersection data (and consequently of the abelian
invariants) is algorithmic from a knot diagram, providing a new
3-dimensional procedure to compute the Alexander polynomial and the
Blanchfield form of a knot.  We describe the algorithm in
Section~\ref{subsection:algorithm}, and we work through a detailed
example in Section~\ref{subsection:examples-accessory-discs-only}.

\subsection{Whitney towers and the Arf invariant}

Recall that we used the Whitney framing to compute the
$\Z[t,t^{-1}]$-valued intersection number of an order one Whitney disc
$D_1^i$ with itself.  In general, an extension of the Whitney framing
to $D_1^i$ may have zeros; the Whitney framing extends to a
non-vanishing section on $D_1^i$ if and only if it agrees with the
unique framing of the normal bundle of~$D_1^i$.  Following common
convention, we call such a Whitney disc \emph{framed}.  A Whitney tower is
\emph{framed} if all the Whitney discs in the tower are framed.

The generic number of zeroes, counted with sign, of an extension of
the Whitney framing to the normal bundle of the Whitney disc is called
the \emph{twisting coefficient}.  If a given Whitney disc is not
framed, by interior twisting we can alter the twisting coefficient by
any multiple of $2$, and whence if the twisting coefficient were even,
we could arrange that the Whitney disc be framed.  This motivates the following definition.

\begin{definition}\label{defn:twisted-disc}
  A Whitney disc is \emph{essentially twisted}
  if its twisting coefficient is odd.
\end{definition}

We recall the definition of the Arf invariant of a knot, in terms of a
Seifert matrix, in
Definition~\ref{defn:arf-invariant-seifert-surface}.  The following
theorem follows from work of Matsumoto, Kirby, Freedman and
Quinn~\cite{Matsumoto:1978-1, Freedman-Kirby:1978},
\cite[Section~10.8]{Freedman-Quinn:1990-1}.  See
also~\cite[Lemma~10]{Conant-Schneiderman-Teichner:2012-3}.

\begin{theorem}[Freedman, Kirby, Matsumoto, Quinn]
  \label{theorem:scheiderman-arf-inv}
  The Arf invariant $\Arf(K)$ vanishes if and only if $K$ is the
  boundary of framed Whitney tower of order two in~$D^4$.
\end{theorem}

In fact, Schneiderman~\cite{Schneiderman:2006-1} also showed that the
Arf invariant is the only obstruction for a knot to bound a framed
(asymmetric) Whitney tower of any given order: a knot which bounds a
framed order two Whitney tower in $D^4$ bounds a framed order~$n$
Whitney tower for all~$n$.

J.~Levine showed that the Arf invariant of a knot, defined in terms of
the Seifert form (recalled in
Definition~\ref{defn:arf-invariant-seifert-surface}), can be computed
in terms of the Alexander
polynomial~$\Delta_K(t)$~\cite[Sections~3.4~and~3.5]{Levine:1966-1}.
He used the fact that the Alexander polynomial can be computed as
$\det(tV-V^T)$, where~$V$ is a Seifert matrix for~$K$.

\begin{theorem}[Levine]
  The Arf invariant $\Arf(K)$ of a knot $K$ satisfies:
  \[\Arf(K) =
  \begin{cases}
	0 & \text{if }\Delta_K(-1) = \pm 1 \mod 8, \\
        1 & \text{if }\Delta_K(-1) = \pm 3 \mod 8.
  \end{cases}
  \]
\end{theorem}

The absolute value of the Alexander polynomial evaluated at $-1$ is
also the order of the homology of the 2-fold branched cover of $K$,
which is a $\Z_{(2)}$-homology circle.  In particular,
$\Delta_{K}(-1)$ is always an odd number.  The Arf invariant measures,
up to a unit, whether $\Delta_K(-1)$ is a square modulo 8.

By combining the two previous theorems, the following is known.

\begin{theorem}[Freedman, Kirby, Levine, Matsumoto, Quinn]
  A knot $K$ bounds a framed Whitey tower of order two if and only if
  $\Delta_K(-1) \equiv \pm 1 \mod{8}$.
\end{theorem}

However the only previously known proof of this result (to the
authors) proceeds by passing via the Seifert form definition of the
Arf invariant.  We give a new, direct, 4-dimensional proof that the
Whitney tower and Alexander polynomial interpretations of the Arf
invariant are equivalent.  More precisely, we show the following.

\begin{theorem}\label{theorem:arf-whitney-alexander-poly}
  Suppose $K$ bounds an order two Whitney tower where~$n$ of the order one Whitney discs are
  essentially twisted.  Then
  \[
  \Delta_K(-1) \equiv
  \begin{cases}
     \pm 1 \mod 8  & \text{if }n \equiv 0 \mod 2,\\
     \pm 3 \mod 8  & \text{if }n \equiv 1 \mod 2.
  \end{cases}
  \]
\end{theorem}

If an order 2 Whitney tower has an even number of essentially twisted
Whitney discs, then it can be modified by geometric moves to a framed
order 2 tower.  This follows easily from
\cite[Theorem~2.15]{Conant-Schneiderman-Teichner:2012-2}; for the
convenience of the reader we sketch the procedure in
Lemma~\ref{lemma:even-twisted-discs-cancel}.  However note that we do
not need this step: the Alexander polynomial conclusion can be drawn
if we have an order two Whitney tower with an even number of
essentially twisted discs.

\subsubsection*{Motivation}

In future work, we hope to describe all nilpotent invariants of links (roughly, invariants carried by duality and the homology of a nilpotent cover) from the intersection theory of an asymmetric Whitney tower for the link.

Of particular interest are the postulated ``higher order
Arf invariants'' of Conant, Schneiderman and 
Teichner~\cite{Conant-Schneiderman-Teichner:2012-2, Conant-Schneiderman-Teichner:2012-3}.  They ask whether a link bounds an asymmetric framed Whitney tower in the 4-ball, and define an obstruction theory involving an algebra of
labelled uni-trivalent trees.  They show that
Milnor's link invariants and the Arf invariant are obstructions to building towers.  Additional non-trivial
trees in their algebra do not correspond to any known invariants,
and may obstruct higher order framed Whitney towers for certain links.  The main examples of these links
are iterated Bing doubles of knots with non-vanishing Arf invariant.
Conant, Schneiderman and Teichner call these invariants the \emph{higher
  order Arf invariants}, and these invariants live either in $\Z_2$ or $0$.  If the higher order Arf invariants were trivial, one would need to add new relations to the tree algebra. We recommend~\cite{Conant-Schneiderman-Teichner:2012-2, Conant-Schneiderman-Teichner:2012-3} for further reading.  It is with this problem in mind that we put such emphasis on giving a new proof of the long-known relationship between Whitney towers and the Arf invariant.

\subsection*{Organisation of the paper}

Section~\ref{section:presentation tower} constructs a presentation
tower for the knot, that is an order two immersed Whitney tower with
the special attributes described in
Definition~\ref{definition:pres-tower}.
Section~\ref{section:intersection-form} gives the statement of our
main technical theorems,
Theorem~\ref{theorem:intersection-form-using-whitney-discs} and
Theorem~\ref{theorem:intersection-form-using-accessory-discs}, on the
structure of the intersection form of the exterior $W$ of an immersed
disc $D_0 \looparrowright D^4$, and the relation of this intersection
form to the combinatorics of Whitney and accessory disc intersections.
Section~\ref{section:Proof} is devoted to the proof of the technical
theorems.  Section~\ref{section:pi-2-is-free} shows that $\pi_2(W)$ is
a free module. Section~\ref{section:construction-of-spheres}
constructs the spheres we use to compute the intersection
form. Section~\ref{section:framings-definitions} gives the precise
definitions of Whitney and accessory framings.
Sections~\ref{section:intersection-of-spheres-0}
through~\ref{section:int-spheres-4} compute the intersections of the
spheres, proving
Theorems~\ref{theorem:intersection-form-using-whitney-discs} and~\ref{theorem:intersection-form-using-accessory-discs}.
Section~\ref{section:homology-boundary} computes the homology of
$\partial W$.
Section~\ref{section:proof-theorem-compute-alex-poly-from-whitney-data}
collates the results of the previous two sections, proving
Theorems~\ref{theorem:compute-alex-poly-from-whitney-data},~\ref{theorem:main-presentation-Blanchfield}  and ~\ref{theorem:main-blanchfield-alexander}, apart from the
Blanchfield form assertions.
Section~\ref{section:examples} gives some
example computations.
Section~\ref{section:defns-of-Arf} recalls, for
completeness, the usual definition of the Arf invariant in terms of
the Seifert form.
Section~\ref{section:proof-Arf-invariant-theorem}
proves Theorem~\ref{theorem:arf-whitney-alexander-poly} relating the
Alexander polynomial at $-1$ to the modulo two count of the number of
twisted Whitney discs.
Section~\ref{section:blanchfield} considers
the Blanchfield form and completes the proof of
Theorems~\ref{theorem:main-presentation-Blanchfield} and
\ref{theorem:compute-alex-poly-from-whitney-data}.

\subsection*{Acknowledgements}

We would like to thank Peter Teichner for several useful suggestions, and in
particular for sharing the idea of the construction of spheres from Whitney
discs.  We also thank Nick Castro, Stefan Friedl, Daniel Kasprowski, Allison
Moore and Rob Schneiderman for their interest and helpful comments. The second
author gratefully acknowledges the support provided by the SFB 1085 Higher
Invariants at the University of Regensburg while on sabbatical in the fall of
2014, and funded by the Deutsche Forschungsgemeinschaft (DFG).

\section{Construction of an order two presentation tower for a knot}\label{section:presentation tower}

We begin with a properly immersed disc $D'_0$ in $D^4$ with boundary a
knot $K \subset S^3$ which has an algebraically vanishing count of
self-intersection points.  This can always be arranged by adding local
cusp singularities to $D_0'$~\cite[p.~72]{Kirby:1989-1}.  Such a disc
induces the zero framing on its boundary~$K$.  In the next two
subsections, we will show how to find a new immersed disc $D_0$,
regularly homotopic to $D'_0$, the complement of which has infinite
cyclic fundamental group.  We will then show how to find order one
Whitney discs $D_1^1,\ldots,D_1^k$, that are potentially twisted, in
the exterior of $D_0$.  Here $D_0$ has $d=2k$ double points.  In our
results relating knot invariants to Whitney towers, we will use
intersection data from the order one Whitney discs, together with data
from additional discs called \emph{accessory} discs.  This will
construct an order two presentation tower for $K$, as promised.

For a double point $p$ of $D_0$, a \emph{double point loop} is a loop
on $D_0$ that leaves $p$ along one sheet and returns along the other,
avoiding all other intersection points.  An \emph{accessory disc} (see
\cite[Section~3.1]{Freedman-Quinn:1990-1}) is a disc in $D^4 \sm \nu D_0$ whose
boundary is a push-off of a double point loop to the boundary
$\partial_+ := \partial (\cl(\nu D_0)) \sm \nu K$ of a neighbourhood
of $D_0$.  By a judicious choice, the push-off can be arranged to be
trivial in $\pi_1(D^4 \sm \nu D_0) \cong \Z$.  It therefore bounds an accessory disc
in~$D^4 \sm \nu D_0$. (See Lemma~\ref{lemma:accessory-discs-exist-part-1} below.)

For each Whitney disc $D_1^i$, pick one of the two intersections
paired by $D_1^i$, and produce an accessory disc $A_1^i$ for this
intersection as above.

\subsection{Fixing the fundamental group}\label{section:fixing-pi-1}

\begin{lemma}\label{lemma:-make-pi-1-Z}
  A properly immersed disc $D'_0$ in $D^4$ with boundary a knot
  $K \subset S^3$ is regularly homotopic to a disc $D_0$ for which
  $\pi_1(D^4 \sm \nu D_0) \cong \Z$.  Moreover, new double points
  support order 1 framed Whitney discs.
\end{lemma}

\begin{proof}
  The idea is to use finger moves, as introduced by Casson in
  \cite{Casson-1986-towers}.  A finger move kills a commutator of the
  form $[g,g^w]$, where $g$ is a meridian of $D'_0$, $w$ is the curve
  the finger pushes along, and $g^w$ means $wgw^{-1}$.

  Apply finger moves to make any pair of meridional loops commute.
  Since meridional loops (finitely) generate the fundamental group,
  the fundamental group $\pi_1(D^4\sm\nu D_0)$ corresponding to the
  new immersed disc $D_0$ is the abelianisation of
  $\pi_1(D^4\sm\nu D_0')$ which is~$\Z$.
\end{proof}

Define $W:= D^4 \sm \nu D_0$ to be the exterior of the immersed disc
$D_0$ produced by Lemma~\ref{lemma:-make-pi-1-Z}.  A consequence of
Lemma~\ref{lemma:-make-pi-1-Z} is the existence of an accessory disc.

\begin{lemma}\label{lemma:accessory-discs-exist-part-1}
  Each double point of $D_0$ has an accessory disc in~$W$.
\end{lemma}

\begin{proof}
 Choose a push-off of a double point loop.  By
  winding the push-off around a meridian to $D_0$ if necessary,
  arrange that the push-off is null-homotopic in $W$.  Here we use
  that $\pi_1(W) \cong \Z$.  A null-homotopy in general position gives
  us an accessory disc as required.  Here we do not impose any framing
  condition on the accessory disc.
\end{proof}

  The same argument applies to the Whitney disc case, showing that any pair of double points with opposite sign admit a (potentially twisted) order one Whitney disc in $W$.

\subsection{Arranging $D_0 \cap D_1 = \emptyset$}
\label{section:arrange-D_1-misses-D_0}

A \emph{Whitney tower of order one} is a properly immersed disc $D_0$
together with Whitney discs $D_1 = D_1^1 \cup \cdots \cup D_1^k$ which
pair up all the double points of~$D_0$.  The Whitney discs are said to
have \emph{order one} (since they pair self-intersections of the order
zero disc.)  We impose nothing about the framing of the Whitney
discs.  We remark that we can indeed arrange each Whitney disc to be
framed, by applying boundary twists, and in this case the tower is
called a \emph{framed Whitney tower of order one}.

In an order one Whitney tower, since a Whitney disc pairs double
points of opposite signs, $D_0$ automatically has vanishing algebraic
self intersection.  Conversely, when $D_0$ is an immersed disc in
$D^4$ with algebraic self-intersection zero then, since $D^4$ is
simply connected, there exist Whitney discs which pair up all the
double points.

Furthermore, for any given order one Whitney tower, we can modify the tower so that the interiors of the order one Whitney discs are disjoint from the order zero disc~$D_0$, as required in the definition of an \emph{order 2
  presentation tower} (Definition~\ref{definition:pres-tower}).
  For the convenience of the reader, we explain the procedure in the next lemma, which is well known to the
experts.

This is a special case of a general result of
Conant-Schneiderman-Teichner, c.f.\
\cite[Proof~of~Lemma~10]{Conant-Schneiderman-Teichner:2012-3}.
However, note that Conant-Schneiderman-Teichner do not need to
actually cancel intersection points geometrically; in their situation
it is enough to pair them up with Whitney discs which admit higher
order intersections only.  For this reason we spell out the details in
our special case.  If one wishes to simply show the existence of an
order two presentation tower, rather than promoting a given order one
Whitney tower, one can choose Whitney discs in the exterior of $D_0$,
as in the remark just after the proof of
Lemma~\ref{lemma:accessory-discs-exist-part-1}.

Everything in 4-manifold topology seems to comes at a price, and in
this case we can arrange the desired disjointness $D_0 \cap D_1 = \emptyset$ at the cost of
allowing twisted Whitney discs.

\begin{lemma}\label{lemma:make-D-0-cap-D-1-empty}
  Let $D_0 \cup D_1$ be an order one Whitney tower, where
  $D_1=D_1^1 \cup \dots \cup D_1^n$.  After performing boundary twists
  on $D_1$, there is a regular homotopy of $D_0$ to an immersed disc
  $D_0'$ which supports an order 2 tower $D_0' \cup D_1'$  where $D_0' \cap \inte D_1' = \emptyset$.
\end{lemma}

\begin{proof}
  A boundary twist~\cite[Section~1.3]{Freedman-Quinn:1990-1} of an
  order one Whitney disc $D_1^i$ adds an intersection point
  $D_1^i \cap D_0$.  Perform boundary twists until all such
  intersection points occur algebraically zero times.  The $D_1^i$ may
  now be twisted (essentially or otherwise).  Pair up the intersection points in $D_0 \cap D_1^i$
  and find Whitney discs $D_2$ for each pair.  These always exist by
  simple connectivity of $D^4$.  However we may have that
  $D_1 \cap D_2$ and $D_0 \cap D_2$ are nonempty.  Push the
  intersections $D_1 \cap D_2$ off $D_2$ over the $D_1$ part of its
  boundary by a finger move.  This creates new $D_1 \cap D_1$
  intersections but we do not mind. Push the intersections
  $D_0 \cap D_2$ off the $D_0$ part of the boundary.  This creates new
  $D_0 \cap D_0$ intersections.  These have to be paired up with a new
  order 1 Whitney disc $D_1^j$.  This is always possible, since the
  new intersections came from a finger move (note that the new disc
  $D_1^j$ is framed).  One has to be careful that the new Whitney arcs
  for the new $D_1^j$ do not intersect the Whitney arcs for $D_2$.
  This can easily be arranged by pushing the boundary arc (see
  \cite[Figures~6,~7~and~8]{Schneiderman:2005-1}), but means that the
  new $D_1^j$ intersects the old $D_1^i$ (the order one disc whose
  intersections with $D_0$ are being paired up by $D_2$).  However new
  $D_1 \cap D_1$ intersections are allowed.  We have now arranged that
  $D_2$ is disjoint from everything.  Therefore we can use it to
  perform the Whitney move.  Push $D_1^i$ across $D_2$.  Any
  self-intersections of $D_2$ result in more $D_1 \cap D_1$
  intersections, but again these are permitted. We have now decreased
  the number of intersection points in $D_0 \cap D_1$ by two, at the
  cost of new intersection points in $D_0 \cap D_0$, $D_1 \cap D_1$,
  potentially twisting a $D_1$ disc, and a new $D_1$ Whitney disc
  which is disjoint from $D_0$.  These are all within our budget.  By
  repeating this process we can therefore arrange that
  $D_1 \cap D_0 = \emptyset$ as claimed.  All the operations apart from the boundary twists are
  regular homotopies on the original discs, together with introducing
  new order 1 Whitney discs to pair up new $D_0 \cap D_0$
  intersections.
\end{proof}

We quickly indicate how to see the following statement, since the
argument of the proof of Lemma~\ref{lemma:make-D-0-cap-D-1-empty} is
pertinent.  We do not need the following lemma but include it for
completeness, since it is closely related to
Theorem~\ref{theorem:arf-whitney-alexander-poly}.

\begin{lemma}\label{lemma:even-twisted-discs-cancel}
  Let $D_0 \cup D_1$ be an order two Whitney tower with an even number of essentially twisted Whitney discs and
   $D_0 \cap \inte D_1= \emptyset$.  Then
  there is a regular homotopy of $D_0$ to a new immersed disc $D_0'$
  which supports a framed Whitey tower of order two $D_0' \cup D_1'$
  with $D_0' \cap D_1' = \emptyset$.
\end{lemma}

\begin{proof}
  For each pair of essentially twisted Whitney discs, perform interior twists so
  that one has twisting coefficient $+1$ and the other has twisting
  coefficient $-1$.  Then perform boundary twists so that both are
  framed.  This introduces a pair of $D_1 \cap D_0$ intersections.
  The proof of
  \cite[Theorem~2.15]{Conant-Schneiderman-Teichner:2012-2} enables us
  to perform regular homotopies so that these arise on the same order~1 Whitney disc. We may then pair them up with an order 2 Whitney
  disc~$D_2$. Now we apply the argument of the proof of
  Lemma~\ref{lemma:make-D-0-cap-D-1-empty} to trade the $D_0 \cap D_1$
  intersections for higher order $D_1 \cap D_1$ intersections and
  potentially new $D_0 \cap D_0$ intersections which are paired by new
  framed order 1 discs.  This produces an order 2 framed Whitney tower
  as claimed.
\end{proof}

\section{The intersection form of an immersed disc exterior in the
  4-ball}
\label{section:intersection-form}

In this section we give the detailed description of the matrices
$\Omega$ and $\Psi$ from the introduction (Theorems~\ref{theorem:compute-alex-poly-from-whitney-data} and~\ref{theorem:main-presentation-Blanchfield} respectively), in terms of intersection
data of the Whitney and accessory discs, and we state our main
technical results, that relate these matrices to the intersection
pairing of an immersed disc exterior.

Suppose that a knot $K$ bounds an order two presentation tower as
constructed in Section~\ref{section:presentation tower}, where the
order zero disc $D_0$ has $d=2k$ self-intersection points.  We may
assume by Lemma~\ref{lemma:-make-pi-1-Z} that
$\pi_1(D^4 \sm \nu D_0) \cong \Z$.  Consider the free module
$\Z[\Z]^{d}$, with basis elements $e_{2i-1}$, $i=1,\dots,k$
corresponding to order one Whitney discs $D_1^1,\dots,D_1^k$ pairing
up the double points, and with the basis elements $e_{2i}$,
$i=1,\dots,k$ corresponding to accessory discs $A^1_1,\dots,A^k_1$
(see \cite[Section~3.1]{Freedman-Quinn:1990-1} and
Lemma~\ref{lemma:accessory-discs-exist-part-1}) for half of the
self-intersections of $D_0$, the double point with a positive sign for
each pair which is paired up by one of the $D_1^i$.

The matrix $\Omega$ described below is hermitian, that is
$\Omega=\overline{\Omega}{}^T$, and defines a pairing
$\Omega \colon \Z[\Z]^d \times \Z[\Z]^d \to \Z[\Z]$.  Here the
overline denotes the involution on the group ring $\Z[\Z]$ defined by
extending $t \mapsto t^{-1}$ linearly.  We abuse notation and conflate
the matrix and the pairing which it determines on $\Z[\Z]^d$.

As before define $W:= D^4 \sm \nu D_0$. Choose a path from a chosen
basepoint of each $D_i^i$, and of each $A_1^j$, to the basepoint of
$W$.  For each intersection point $q$ involving $D_1^i$, choose a path
from $q$ to the basepoint of $D_1^i$, inside $D_1^i$ and missing all
double points.  Similarly for the $A_1^i$.  For each intersection
point in $D_1^i \cap D_1^j$, $D_1^i \cap A_1^j$ and
$A_1^i \cap A_1^j$, there is an associated element $\pm t^{\ell}$ of
$\pi_1(W) \cong \Z$, defined by considering the usual concatenation of
paths.  By summing over such intersection points we obtain an element
$p(t)$ of $\Z[\Z]$.  Let $p_{rs}(t)$ be the polynomial associated to
the pair $(e_r,e_s)$.  Note that $p_{rs}(t) = \ol{p_{sr}(t)}$.  When
$i=j$, we abuse notation and use $D_1^i \cap D_1^i$ and
$A_1^i \cap A_1^i$ for the double point set of the immersion.  Here
there is an indeterminacy in $p_{rr}(t)$, up to $t=t^{-1}$, due to a
lack of ordering of sheets at an intersection point.  However this
will not affect the outcome of the computation, so we may make any
choice of ordering.

\subsection{Precise description of the matrix $\Omega$}

The $(r,s)$-entry $\Omega_{rs}$ corresponds to intersection data
involving the discs associated to the pair $(e_r,e_s)$ as given below.
The order of the pair matters since $\Omega_{rs}=\ol{\Omega}_{sr}$.
Define $z:= (1-t)(1-t^{-1})$.

\begin{itemize}
\item For $r \neq s$ and $\{r,s\} \neq \{2i-1,2i\}$,  $\Omega_{rs} = zp_{rs}(t)$.
\item For $\{r,s\} = \{2i-1,2i\}$ for some $i$, $\Omega_{rs} = zp_{rs}(t) + 1$, where
  $p_{rs}(t)$ is computed from intersection points $D_1^i \cap A_1^i$.
\item When $r=s=2i-1$,
  $\Omega_{rs} = zp_{rr}(t) + \ol{zp_{rr}(t)} + za_i$ where
  $p_{rr}(t)=p_{ss}(t)$ arises from the self intersection points of
  $D_1^i$, and $a_i \in \Z$ is the twisting of the Whitney framing
  relative to the disc framing for $D_1^i$.
\item When $r=s=2i$,
  $\Omega_{rs} = zp_{ss}(t) + \ol{zp_{ss}(t)} + zb_i + 1$ where
  $p_{ss}(t) = p_{rr}(t)$ arises from the self intersection points of
  $A_1^i$, and $b_i \in \Z$ is the twisting of the accessory framing
  relative to the disc framing for $A_1^i$.
\end{itemize}

The first and last cases only are relevant to $\Psi$ from Theorem~\ref{theorem:main-presentation-Blanchfield}.
Precise definitions of the Whitney and accessory framings are given in
Section~\ref{section:framings-definitions}.

For practical purposes it is not always convenient to have the
accessory disc correspond to a double point with positive intersection
sign.  If we use a double point with negative sign, then replace the $+1$ in $\Omega_{2i,2i}$ entry in the last bullet point with a $-1$.

\subsection{Structure of the intersection form of $W$}
  The following is one of our main technical results.

\begin{theorem}
  \label{theorem:intersection-form-using-whitney-discs}
  Suppose that $D_0$ has $d=2k$ double points, and $D_1^i$,
  for $i=1,\dots,k$, are order one Whitney discs pairing up the double
  points of $D_0$, whose interiors are disjoint from~$D_0$.  Let
  $A_1^i$, for $i=1,\dots,k$, be an accessory disc for the $2i$-th double
  point, where the double points are ordered so that even numbered
  points have positive sign.  Then we have the following:
  \begin{enumerate}
  \item\label{item:1-int-form-thm} The homotopy group $\pi_2(W)$ is a
    free $\Z[\Z]$ module of rank $d$.
  \item\label{item:2-int-form-thm} There is a linearly independent set
    $\{S_i\}$ of immersed 2-spheres which generate a free submodule
    $F$ of $\pi_2(W)$ of rank $d$ on which the equivariant
    intersection form $\lambda \colon F \times F \to \Z[\Z]$ can be
    written as $z(X + (zY + \ol{zY}^T))$
    where $X$ is a block diagonal sum of $k$ copies of the form
    \[
      \begin{bmatrix}
        za_i & 1 \\ 1 & 1+ zb_i
      \end{bmatrix}
    \]
    with $a_i,b_i \in \Z$, and $Y$ is an upper triangular $d \times d$
    matrix.
  \item\label{item:3-int-form-thm} The $S_i$ form a basis for
    $\pi_2(W) \otimes_{\Z[\Z]} \Z$.
  \item\label{item:4-int-form-thm} The coefficients $a_i,b_i \in \Z$
    in the $i$th $2 \times 2$ block diagonals of $X$ are the twisting
    numbers of the $i$th Whitney disc $D^i_1$ and the $i$th accessory disc
    $A^i_1$ respectively.
  \item\label{item:5-int-form-thm} The coefficients of $Y$ are the
    $\Z[\Z]$-twisted intersection numbers and self-intersection
    numbers of the $D_1^i$ and the $A_1^i$.
  \end{enumerate}
\end{theorem}

Comparing the matrix $\Omega$ defined above with the matrix
of the intersection form of~$W$, we have $\lambda = z \Omega$.

The proof of this theorem will take the entire next section.  In the
course of the proof we explicitly construct immersed 2-spheres $S_i$
which represent elements of $\pi_2(W)$ and compute the intersection
form using these explicit elements and intersections between Whitney
discs and accessory discs.

It is quite possible that $F = \pi_2(W)$, however we are only able to
prove this in the special case that $D_0$ arises from crossing
changes; see Lemma~\ref{lemma:spheres-give-basis}.

We have another version which only uses accessory discs, and which is
used to deduce Theorem~\ref{theorem:main-presentation-Blanchfield}.
For the purpose of deducing
Theorem~\ref{theorem:main-presentation-Blanchfield} we give the
explicit statement.

\begin{theorem}
  \label{theorem:intersection-form-using-accessory-discs}
  Suppose that $D_0$ has $d$ double points and $A^i$ are accessory
  discs $(i=1,\dots,d)$ whose interiors are disjoint from $D_0$.
    Then we have the
  following:
  \begin{enumerate}
  \item The homotopy group $\pi_2(W)$ is a free $\Z[\Z]$ module of
    rank $d$.
  \item There is a linearly independent set $\{S_i\}$ of immersed
    2-spheres which generate a free submodule $F$ of $\pi_2(W)$ of
    rank $d$ on which the equivariant intersection form
    $\lambda \colon F \times F \to \Z[\Z]$ can be written as
    $z(X + (zY + \ol{zY}^T))$
    where $X$ is a diagonal matrix with entries $ \pm 1 + zb_i$, with
    $b_i \in \Z$, and $Y$ is an upper triangular $d \times d$
    matrix.
  \item The $S_i$ form a basis for $\pi_2(W) \otimes_{\Z[\Z]} \Z$.
  \item The coefficients $b_i \in \Z$ in $X$ are the twisting numbers
    of the $i$th accessory disc $A_i$, and the $\pm 1$ is determined
    by the sign of the $i$th double point.
  \item The coefficients of $Y$ are the $\Z[\Z]$-twisted intersection
    numbers and self-intersection numbers of the $A_i$.
  \end{enumerate}
\end{theorem}

Compare this with the matrix $\Psi$ from the introduction to observe
that $\lambda = z \Psi$.
Both sets of spheres from the above two theorems arise from ambient surgery on a basis of $H_2(W;\Z)$
comprising Clifford tori of the double points, as we will see in
Section~\ref{section:construction-of-spheres}.  Restricting the proof
of Theorem~\ref{theorem:intersection-form-using-whitney-discs} to the
accessory discs only gives the proof of
Theorem~\ref{theorem:intersection-form-using-accessory-discs}.
Therefore we focus on
Theorem~\ref{theorem:intersection-form-using-whitney-discs}.

\section{Proofs of the intersection form
  Theorems~\ref{theorem:intersection-form-using-whitney-discs}
  and~\ref{theorem:intersection-form-using-accessory-discs}}
\label{section:Proof}

\subsection{The second homotopy group of $W$ is a free module}
\label{section:pi-2-is-free}

In this subsection we prove the following.

\begin{lemma}\label{lemma:pi-2-free-module}
  The homotopy group $\pi_2(W)$ is a free $\Z[\Z]$ module.
\end{lemma}

\begin{proof}
   Let $R:=\Z[\Z]$.  Since $\pi_1(W) \cong \Z$, we have $H_1(W;R) = 0$ and
  $H_2(W;R) \cong \pi_2(W)$.  We therefore need to show that
  $H_2(W;R)$ is a free module, which follows from general arguments
  on 4-manifolds with fundamental group~$\Z$.  The relative cohomology
  group $H^2(W,\partial W;R)$ can be computed using the universal
  coefficient spectral sequence
  \[
    E_2^{p,q} = \Ext^p_{R}(H_q(W,\partial W;R), R) \Longrightarrow
    H^n(W,\partial W;R),
  \]
  where the differential $d_r$ on $E_r^{p,q}$ has degree $(r,1-r)$
  (see e.g.~\cite[Theorem~2.3]{Levine-77-knot-modules}).  First, from
  the long exact sequence of a pair and from $H_1(W;R) = 0$, it
  follows that $H_1(W,\partial W;R) = 0$.  From this and from
  $H_0(W,\partial W;R) = 0$, it follows that the only nontrivial term
  on the line $p+q=2$ on the $E^2$ page is
  $E^{0,2}_2= \Hom_{R}(H_2(W,\partial W;R),R)$.  Since
  $H_1(W,\partial W;R) = 0$ and since $R$ has homological dimension
  $2$ (or since $H_0(W,\partial W;R)=0$), the differentials
  $d_r^{0,2}$ at $E^{0,2}_r$ for $r\ge 2$ are into trivial codomains
  and thus trivial.  Therefore we deduce that
  \[
    H^2(W,\partial W;R) \cong  \Hom_{R}(H_2(W,\partial W;R),R).
  \]
  This is a free module, since $\Hom_{R}(A,R)$ is free for any
  $R$-module $A$, by~\cite[Lemma~3.6]{Kaw86} or~\cite[Lemma~2.1]{Borodzik-Friedl-2013-alg-unknotting}.
    Therefore $H^2(W,\partial W;R) \cong H_2(W;R)$ is free as claimed.
\end{proof}

\subsection{Construction of spheres in $\pi_2(W)$}
\label{section:construction-of-spheres}

We proceed to construct explicit elements of $\pi_2(W)$ whose
intersection data can be computed in terms of intersection and
twisting data for the discs $D_1$ and $A_1$.

Consider the Clifford torus for a self-intersection point of~$D_0$.  A
neighbourhood of a self-intersection point is homeomorphic to $\R^4$,
in which the two intersecting sheets sit as $\R^2 \times \{0\}$ and
$\{0\} \times \R^2$.  The Clifford torus $T:= S^1 \times S^1 \subset
\R^2 \times \R^2 \cong \R^4$ is shown in a 5-still movie diagram in
Figure~\ref{figure:clifford-torus}.  We may assume that $T$ lies in
$\partial W$.  We will call the curves $S^1\times *$ and $*\times S^1$ ($*\in S^1$),
which are meridians of the two sheets, the \emph{standard basis curves} of~$T$.

\begin{figure}[t]
\begin{center}
\begin{tikzpicture}[scale=.78]
\node[anchor=south west,inner sep=0,scale=.78] at (0,0)
{\includegraphics[scale=0.45]{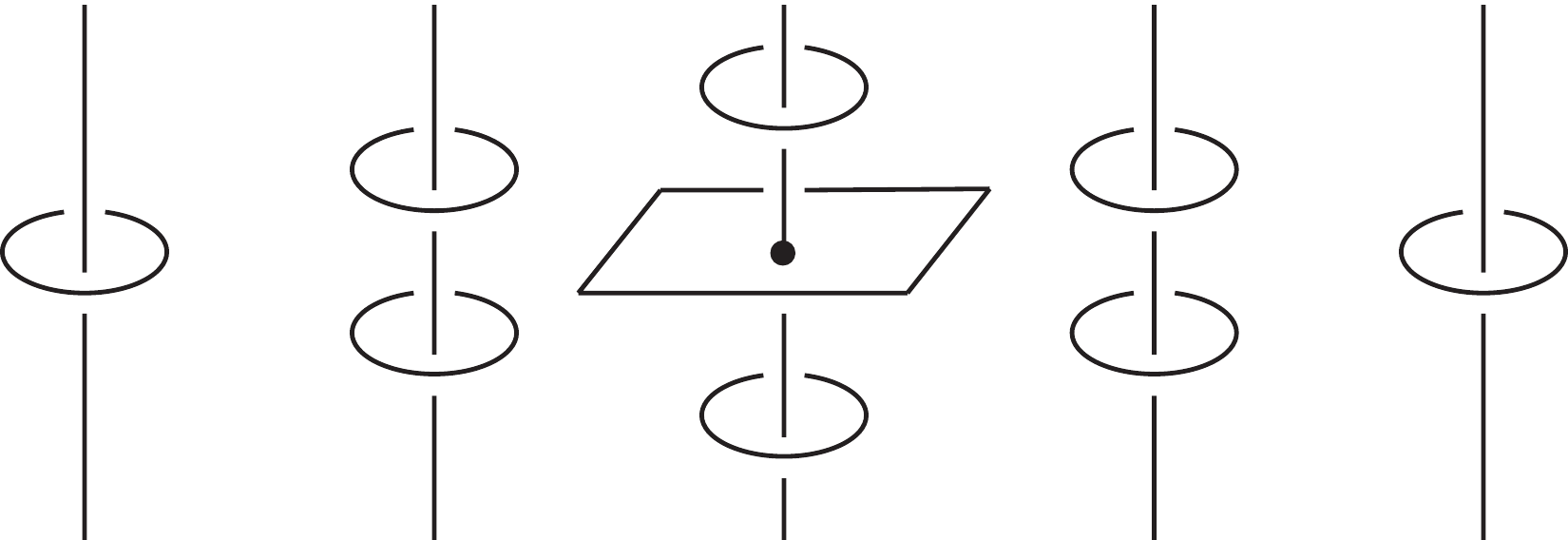}};
\small
\node at (1.53,2.3) {\em T};
\node at (12.75,2.3) {\em T};
\node at (4.33,3.03) {\em T};
\node at (4.33,1.7) {\em T};
\node at (10.1,1.7) {\em T};
\node at (10.1,3.03) {\em T};
\node at (7.16,3.63) {\em T};
\node at (7.16,.98) {\em T};
\end{tikzpicture}
\end{center}
\caption{A Clifford torus $T$ in the neighbourhood of an intersection
  point of two planes in
  $\mathbb R^4 \cong \mathbb R^3 \times \mathbb R$.  One of the planes
  lies in $\mathbb R^3 \times \{0\}$, while the other intersects each
  slice $\mathbb R^3 \times \{\text{pt}\}$ in a line.}
\label{figure:clifford-torus}

\end{figure}
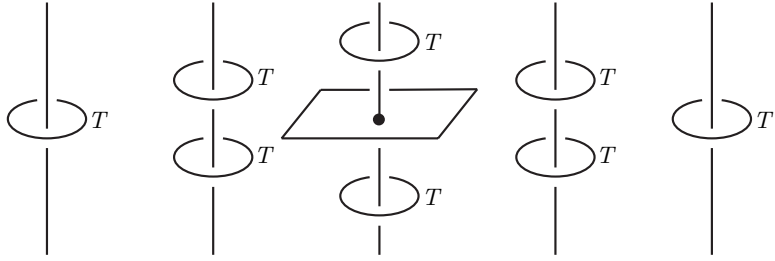

We describe the basic construction of a sphere $S_{2i}$ using an
accessory disc $A_1^i$ for the double point $p$.  The authors learnt
this construction from Peter Teichner.  We will postpone detailed
discussion of framing issues for later computations, for now
contenting ourselves with conveying the main idea of the construction.
We may modify the construction later by inserting interior or boundary
twists into the procedure, in order to arrange that our spheres have
framed normal bundles.

Consider a double point loop $\a$ on $D_0$, and consider the normal
circle bundle to $D_0$ restricted to $\a$.  This defines a torus in
$D^4$.  The intersection of this torus with~$W = D^4 \sm \nu D_0$
defines $N := (\a \times S^1) \cap \partial W$, which is the image of
the map of an annulus into~$W$.  The boundary of $N$ is the two
generating curves on $T$ for $H_1(T;\Z)$.  The boundary of $N$ is thus
a wedge $S^1 \vee S^1$, since the standard basis curves of $T$
intersect in a single point.  The part of $N$ which lies in a $D^4$
neighbourhood of the intersection point is shown in
Figure~\ref{figure:clifford-torus-plus-n}.

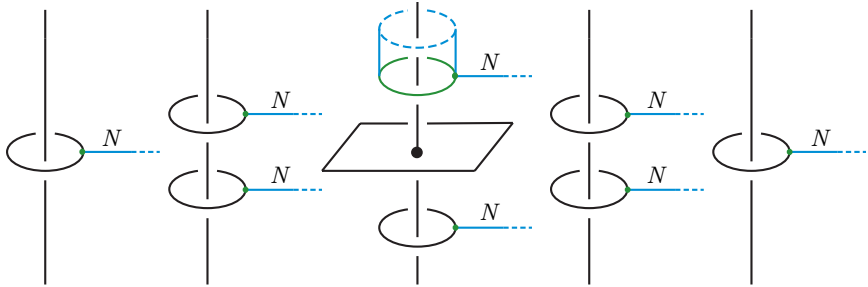
\begin{figure}[H]
\begin{center}
\begin{tikzpicture}[scale=.9]
\node[anchor=south west,inner sep=0,scale=.9] at (0,0)
{\includegraphics[scale=0.45]{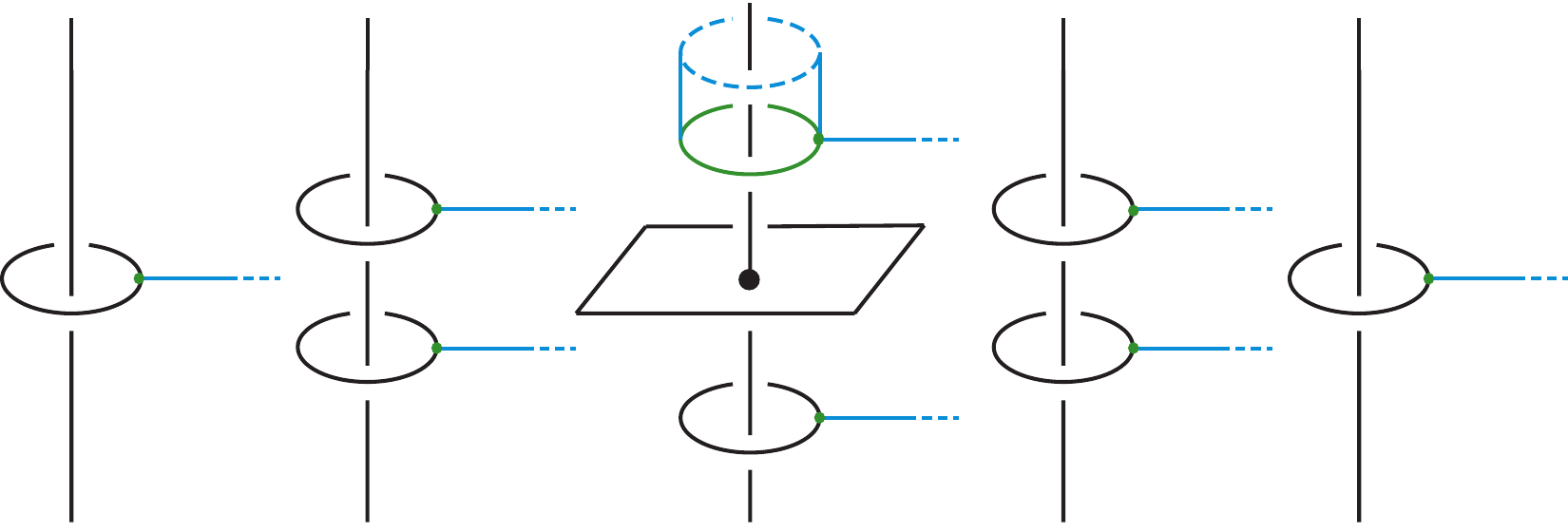}};
\small
\node at (1.53,2.25) {\em N};
\node at (11.9,2.25) {\em N};
\node at (4,2.8) {\em N};
\node at (4,1.7) {\em N};
\node at (9.5,1.7) {\em N};
\node at (9.5,2.8) {\em N};
\node at (7.06,3.36) {\em N};
\node at (7.06,1.16) {\em N};
\end{tikzpicture}
\end{center}
\caption{A Clifford torus $T$ together with the part of the
  annulus~$N$ which lies in a~$D^4$ neighbourhood of the double
  point.}
\label{figure:clifford-torus-plus-n}
\end{figure}

We perform a two step ambient surgery
process.  First use two push-offs of the accessory disc $A^i_1$, which
we denote by $A_{\pm}$, to surger $N$ into a disc  $D = (N \sm \nu A_1^i) \cup A_+ \cup A_-$.  The
boundary of this new disc $D$ is a $(1,1)$ curve on $T$; that is, it
represents the sum of a meridian and a longitude in $H_1(T;\Z) \cong
\Z \oplus \Z$.  In Figure~\ref{figure:surgeries-accessory}, a
schematic of the annulus $N$ is shown, before and after surgery on it
has been performed using $A_+$ and $A_-$ to convert $N$ into the disc
$D$.  We also show the attaching of this apparatus to the Clifford
torus $T$ in Figure~\ref{figure:surgeries-accessory}.  Next, use two
push-offs of $D$ to surger $T$ into an immersed sphere~$S_{2i}$.

\begin{figure}[t]
\begin{center}
\begin{tikzpicture}[scale=.8]
\node[anchor=south west,inner sep=0,scale=.8] at (0,0){\includegraphics[scale=0.5]{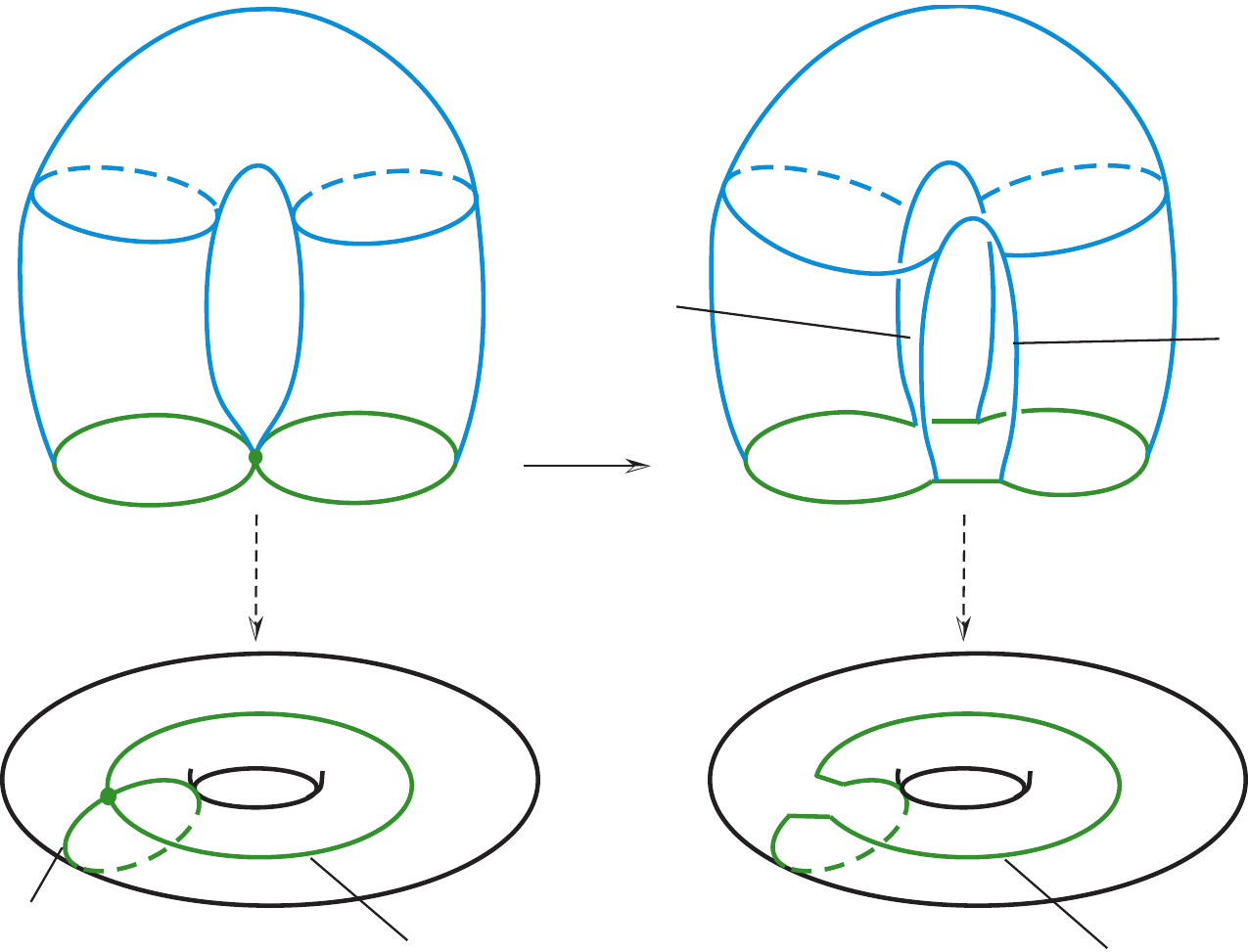}};
\small
\node at (1,8) {$N$};
\node at (2.23,5.66) {$A_1^i$};
\node at (5.62,5.6) {$A_{+}$};
\node at (10.82,5.3)  {$A_-$};
\node at (.3,2.3)  {$T$};
\node at (9.53,8)  {$D$};
\node at (4.3,-.2)  {standard basis curve};
\node at (1,.22)  {standard basis curve};
\node at (10.17,-.23)  {$(1,1)$-curve};
\node at (8.28,7.4)  {$N \sm \nu A_1^i$};
\node at (9,3.35) {attach};
\node at (2.9,3.35)  {attach};
\node at (10.6,2.3)  {$T$};
\node at (5.14,3.7)  {\begin{tabular}{c}surgery\\on $N$\end{tabular}};
\end{tikzpicture}
\end{center}
\caption{Surgery on $N$ using $A_{\pm}$, and the attaching of the resulting disc $D$ to the Clifford torus~$T$.}
\label{figure:surgeries-accessory}
\end{figure}
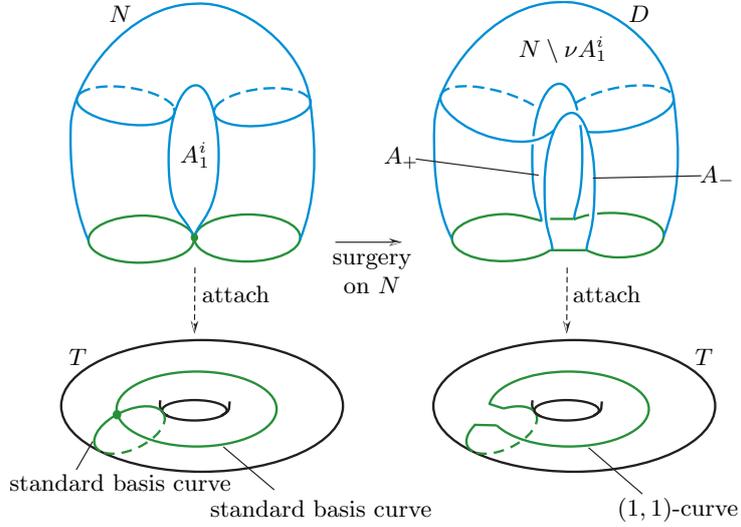

For Theorem~\ref{theorem:intersection-form-using-accessory-discs} this
describes the construction of our entire set of spheres~$\{S_i\}$.
For Theorem~\ref{theorem:intersection-form-using-whitney-discs}, this
creates half of our spheres: use this construction to produce a sphere
from the Clifford torus of one double point in each pair which is
paired up by a Whitney disc.  Recall that we use the double point with
positive sign and recall that $d=2k$.  So we have created spheres
$S_{2i}$ for $i=1,\dots,k$.  For the other spheres, which will form the
other half of our set of spherical elements of $H_2(W;\Z[\Z])$, we
will use the Whitney discs as below.

Let $p_1$, $p_2$ be two double points of $D_0$ which have opposite
intersection signs and which are paired up by an order one Whitney
disc $D_1^i$.   Let $T_1$ and $T_2$ be the Clifford tori for the double points $p_1$ and
$p_2$ respectively.  Let $\a$ be the Whitney circle: a curve which
goes from $p_1$ to $p_2$ on $D_0$, changes sheets, and then returns to
$p_1$ on the opposite sheet to the sheet it left on.  Write $\a = \a_1
\cup \a_2$, dividing $\a$ into two Whitney arcs by cutting at $p_1$
and $p_2$.

Define two annuli in a similar manner to above.  Take the normal
circle bundle to $\a_i$ and consider its intersection with $\partial
W$.  We obtain $N_i := (\a_i \times S^1) \cap \partial W$.  The boundary of $N_1$ is a standard basis curve on
$T_1$ which we shall call the \emph{meridian} of $T_1$, together with
a standard basis curve of $T_2$ which we shall call the
\emph{meridian} of $T_2$.  The boundary of $N_2$ are other standard
basis curves, which we shall call the \emph{longitudes} of $T_1$
and~$T_2$.  A movie of two Clifford tori, the annuli $N_1$ and $N_2$,
and the Whitney disc $D_1^i$ is shown in
Figure~\ref{figure:whitney-disc-plus-n}.  In this figure, the past and
future pictures are drawn only once, since the situation is symmetric
about the zero time slice, $\text{time}=0$.

\begin{figure}[t]
\begin{center}
\begin{tikzpicture}
\node[anchor=south west,inner sep=0] at (0,0){\includegraphics[scale=.45]{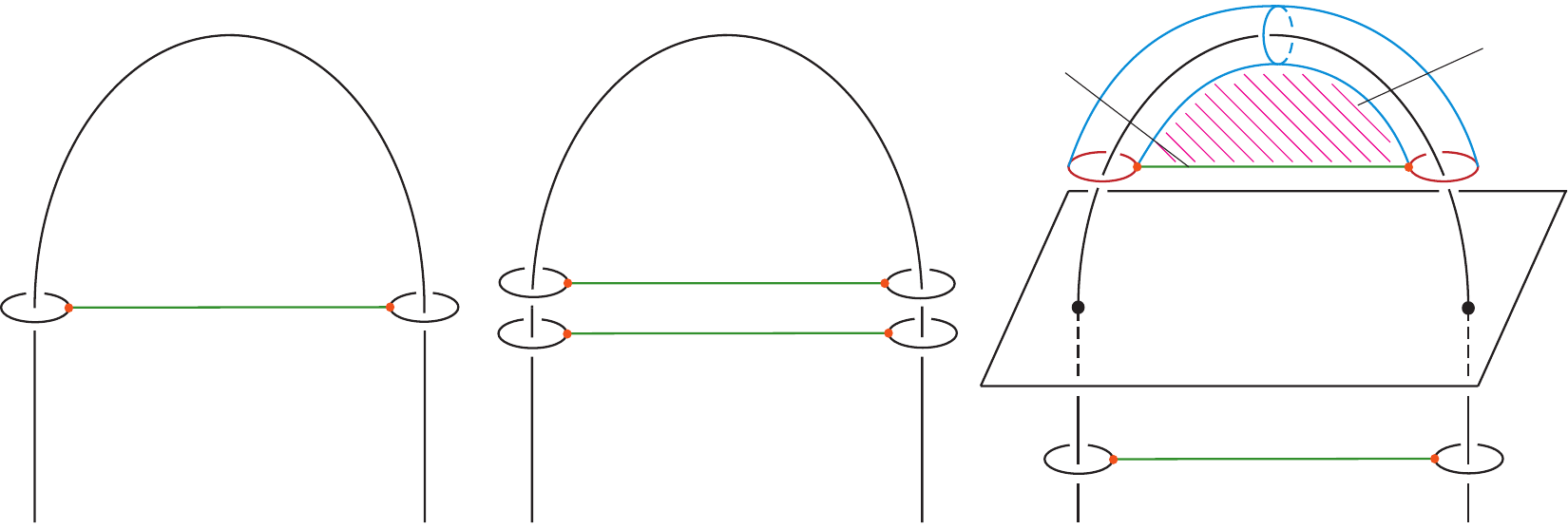}};
\small
\node at (1.8,1.54)  {\smaller $N_1$};
\node at (5.8,1.38)  {\smaller $N_1$};
\node at (5.8,2.21)  {\smaller $N_1$};
\node at (8.29,3.7)  {\smaller $N_1$};
\node at (10.2,.33)  {\smaller $N_1$};
\node at (.78,1.98)  {\smaller $T_1$};
\node at (4.78,2.21)  {\smaller $T_1$};
\node at (8.29,3)  {\smaller $T_1$};
\node at (8.28,.37)  {\smaller $T_1$};
\node at (4.78,1.38)  {\smaller $T_1$};
\node at (7,1.38)  {\smaller $T_2$};
\node at (7,2.21)  {\smaller $T_2$};
\node at (12.1,3)  {\smaller $T_2$};
\node at (12.23,.37)  {\smaller $T_2$};
\node at (2.98,1.98)  {\smaller $T_2$};
\node at (1.8,-.3)  {\smaller $\text{time}= \pm \eps$};
\node at (5.8,-.3)  {\smaller $\text{time}= \pm \frac{\eps}{2}$};
\node at (10.2,-.3)  {\smaller $\text{time}= 0$};
\node at (12.15,3.925){\smaller $D^i_1$};
\end{tikzpicture}
\end{center}
\caption{A picture in $\R^4 \cong \R^3 \times \R$ of a model for two
  intersection points, paired up with a Whitney disc, together with
  their Clifford tori $T_1$ and $T_2$ and the annuli $N_1$ and
  $N_2$. The last $\R$ coordinate is the time.  The future and the
  past are drawn in the same pictures, to avoid repetition.  Note that
  this is only a model.  In reality, since the Whitney disc may not be embedded, all these surfaces may not be
  contained in one contractible open neighbourhood.}
\label{figure:whitney-disc-plus-n}
\end{figure}
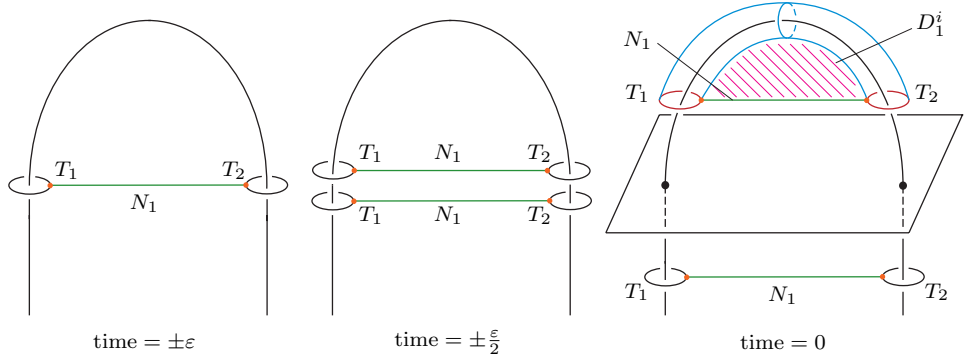

Now we have a three step process.  First use two push-offs $N_1^{\pm}$
of $N_1$ to perform surgery on $T_1$ and $T_2$ to join them into one
bigger torus
\[
T_{12} := N_1^- \cup N_1^+ \cup \cl(T_1 \sm (S^1 \times D^1))\cup \cl(T_2 \sm (S^1 \times D^1)).
\]
Next use two push-offs $(D_1^i)_{\pm}$ of the Whitney disc to convert
$N_2$ into a disc
\[
C:= \cl(N_2 \sm (\a_2 \times D^1)) \cup (D_1^i)_+ \cup (D_1^i)_-.
\]
Here we abuse notation and also denote the push-off of $\a_2$ onto
$N_2$ along $D_1^i$ by~$\a_2$.

Recall that the boundary of $N_2$ was a longitude of $T_1$ and a
longitude of $T_2$.  These longitudes have been cut by the surgery
which converted $T_1 \cup T_2$ into $T_{12}$.  They can be joined by a
pair of arcs, $\a_1^+$ in $N_1^+$ and $\a_1^-$ in $N_1^-$, to create a
longer loop which is a longitude of $T_{12}$, and is also the boundary
of $C$.  The final step is to use two push-offs of $C$ to perform
surgery on $T_{12}$ and create the desired sphere $S_{2i-1}$.  The
schematic arrangement of the constituent parts of $S_{2i-1}$ are shown
in Figure~\ref{figure:whitney-surgery}.

\begin{figure}[H]
\begin{center}
\begin{tikzpicture}[scale=.7]
\node[anchor=south west,inner sep=0,scale=.7] at (0,0){\includegraphics[scale=0.6]{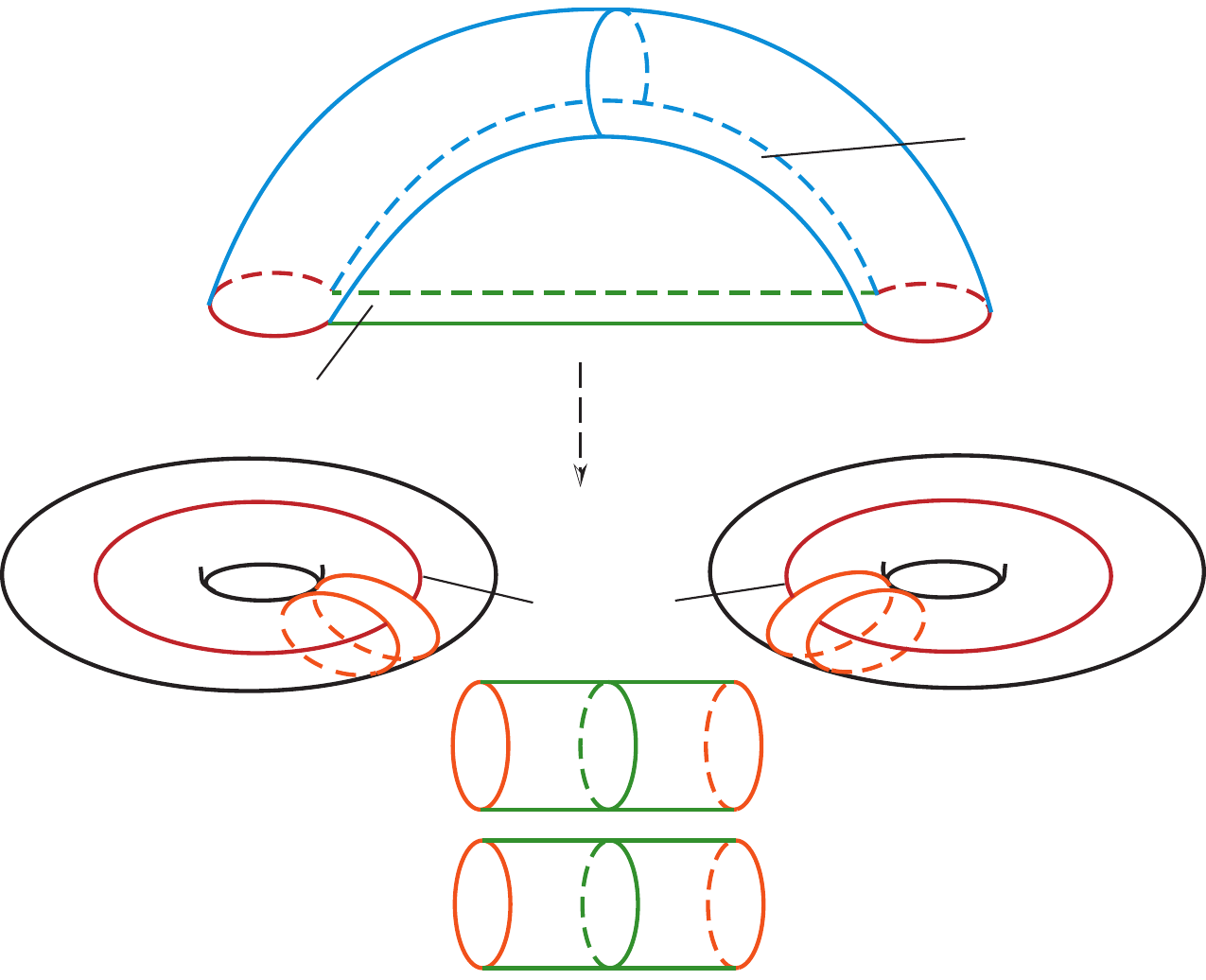}};
\small
\node at (4.5,2.5)  {$N_1^+$};
\node at (4.5,.82)  {$N_1^-$};
\node at (.35,5.4)  {$T_1$};
\node at (12.7,5.4)  {$T_2$};
\node at (9,10) {$C$};
\node at (1.5,9.3)  {$N_2 \sm (D^1 \times D^1)$};
\node at (3.6,6.2)  {$(D_1^i)_+$};
\node at (11.2,9.05)  {$(D_1^i)_-$};
\node at (8,6.1)  {attach $C$ to $T_{12}$};
\node at (6.5,4)  {\begin{tabular}{c}cut \\$\partial N_2$\end{tabular}};
\end{tikzpicture}
\end{center}
\caption{Schematic diagram of the construction of the disc $C$ from
  surgery on $N_2$ using $D_1^i$, and the construction of $T_{12}$
  from the Clifford tori $T_1$ and $T_2$ and two parallel copies
  $N_1^{\pm}$ of the annulus $N_1$.}
\label{figure:whitney-surgery}
\end{figure}
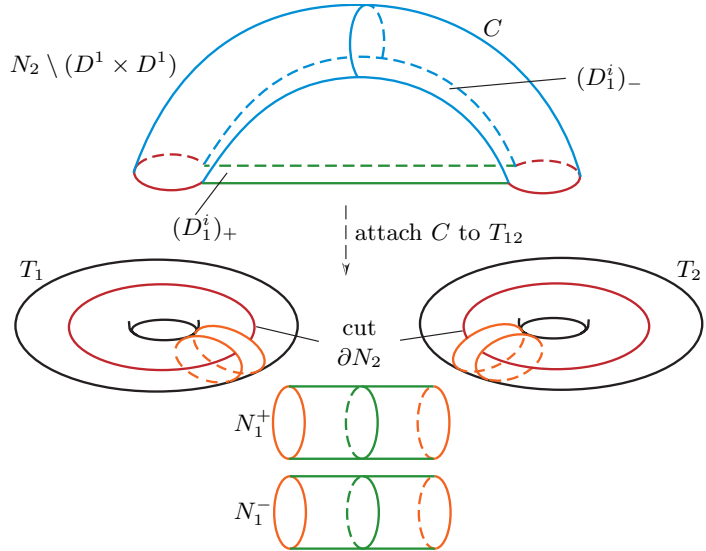

This completes our description of the spheres $S_i$, for
$i=1,\dots,2k=d$.  Recall that we called the submodule in $\pi_2(W)$
they generate~$F$.  Next we will show that $F$ and $\pi_2(W)$ have the
same rank, which is equal to the number of double points of $D_0$.

\begin{lemma}\label{lemma:rank-is-d}
  Both $H_2(W;\Z[\Z])\cong \pi_2(W)$ and its submodule $F$ are free
  $\Z[\Z]$-modules of rank~$d$.
\end{lemma}

Together with Lemma~\ref{lemma:pi-2-free-module}, this proves
(\ref{item:1-int-form-thm}) of
Theorems~\ref{theorem:intersection-form-using-whitney-discs}
and~\ref{theorem:intersection-form-using-accessory-discs}.

\begin{proof}
  The fact that $\pi_1(W) \cong \Z$ is crucial for this proof.  By
  Lemma~\ref{lemma:pi-2-free-module}, $H_2(W;\Z[\Z])$ is a free
  module, so is isomorphic to $\Z[\Z]^\delta$ for some $\delta$.

  \begin{claim*}
    $H_2(W;\Z[\Z]) \otimes_{\Z[\Z]} \Z \cong H_2(W;\Z)$.
  \end{claim*}

   We use the universal
  coefficient spectral sequence for homology~\cite[Theorem~5.6.4]{Weibel94} 
  \[
    E_2^{p,q} = \Tor_p^{R}(H_q(W;\Z[\Z]),\Z) \Longrightarrow
    H_n(W;\Z)
  \]
  to compute $H_2(W;\Z)$ from $H_*(W;\Z[\Z])$.  Here the differential
  $d_r$ has degree $(-r,r-1)$.  The only nontrivial $E^2$ term on the
  line $p+q=2$ is $E^2_{2,0}= H_2(W;\Z[\Z]) \otimes_{\Z[\Z]} \Z$,
  since $H_1(W;\Z[\Z]) = 0$ and $H_0(W;\Z[\Z]) \cong \Z$ admits a
  length one projective $\Z[\Z]$ module resolution
  $\Z[\Z] \xrightarrow{t-1} \Z[\Z] \to \Z$.  The differentials $d^r$
  into $E^r_{2,0}$ ($r\ge 2$) have trivial domains and thus are trivial,
  since $H_1(W;\Z[\Z]) \cong 0$ and since $\Z[\Z]$ has homological
  dimension two.  This completes the proof of the claim.

  Therefore $H_2(W;\Z) \cong \Z[\Z]^\delta \otimes_{\Z[\Z]} \Z \cong \Z^\delta$.
  Now we have a second claim:

  \begin{claim*}
    The second homology is $H_2(W;\Z) \cong \Z^d$, generated by the spheres $S_i$.
  \end{claim*}

   Note that the claim proves (\ref{item:3-int-form-thm}) of
  Theorems~\ref{theorem:intersection-form-using-whitney-discs}
  and~\ref{theorem:intersection-form-using-accessory-discs}.  Assuming
  the claim it follows from $\Z^\delta \cong \Z^d$ that $\delta=d$.
  It also follows that the spheres generating $F$ define linearly
  independent elements of $\pi_2(W) \cong H_2(W;\Z[\Z])$.  To see
  this, note that each sphere $S_i$ lifts to a nontrivial element of
  $H_2(W;\Z[\Z])$; let $\Z[\Z]^d \to H_2(W;\Z[\Z])\cong \Z[\Z]^d$ be
  the homomorphism sending the $i$th basis to $[S_i]$, and let $P(t)$
  be the associated square matrix over~$\Z[\Z]$.  The claim implies
  that $\det P(1)=\pm1$.  It follows that $\det P(t)\ne 0$, that is,
  $\Z[\Z]^d \to H_2(W;\Z[\Z])$ is injective.  So $F$ has rank~$d$.

  It remains to prove the claim that $H_2(W;\Z)
  \cong \Z^d$.  Recall that $\partial_+=\cl(\partial(\nu D_0)
    \sm \nu K)$.  Let $\partial_- = \nu K$.  We have:
  \begin{align*}
    & H_2(W;\Z) \cong H_3(D^4,W;\Z) && \text{by the long exact sequence for $(D^4,W)$,} \\
    & \qquad\cong H_3(\nu D_0, \partial_+;\Z) && \text{by excision,} \\
    & \qquad\cong H^1(\nu D_0, \partial_-;\Z) && \text{by duality,} \\
    & \qquad\cong H^1(D_0,\partial D_0;\Z) \cong H_1(D_0;\Z) = \Z^d, &&
    \text{generated by the double point loops.}
  \end{align*}
  It follows that the Clifford tori, which are dual to the double
  point loops, form a basis for~$H_2(W;\Z)$.  The Clifford tori, after a
  basis change, are homologous to the spheres $S_i$, since the $S_i$
  are obtained from surgery on (linear combinations of) the Clifford
  tori.  This completes the proof of the claim and therefore of
  Lemma~\ref{lemma:rank-is-d}.
\end{proof}

\begin{remark}
  In the case of accessory spheres only, the final basis change is not
  required.  Also, note that unfortunately we do not know that
  $F = \pi_2(W)$, only that the two are both free modules of the same
  rank and that the generators of $F$ give a basis over $\Z$.
  Therefore, choosing a basis for $\pi_2(W)$ and representing the
  generators of $F$ as vectors, and then making these vectors the
  columns of a matrix, yields a matrix $P(t)$ which augments to be
  unimodular.  This matrix appeared in the proof of linear
  independence above and it will appear in the proofs in
  Sections~\ref{section:proof-theorem-compute-alex-poly-from-whitney-data},
  \ref{section:proof-Arf-invariant-theorem}
  and~\ref{section:blanchfield}.  In the special case that $D_0$
  arises from crossing changes, we will see in
  Lemma~\ref{lemma:spheres-give-basis} that $F=\pi_2(W)$.
\end{remark}

\subsection{Definitions of Whitney and accessory framings}\label{section:framings-definitions}

In this section we recall the precise definition of the Whitney
framing of the boundary of a Whitney disc.  Note that a normal bundle to a surface in
4-dimensional space has 2-dimensional fibre.  An orientation of the
surface and an orientation of the ambient space determines an
orientation of the normal bundle.  Thus a single nonvanishing vector
field in the normal bundle of a surface determines two nonvanishing
vector fields, up to homotopy, and therefore a framing.  The second
vector is chosen so as to be consistent with the orientations.

\begin{definition}[Whitney framing]\label{defn:whitney-framing}
  Suppose that we have two surfaces, or two sheets of the same
  surface, $\Sigma_1$ and $\Sigma_2$, intersecting in two points $p$
  and $q$ of opposite signs.  Let $\gamma_i$ be an arc on $\Sigma_i$
  between $p$ and $q$, such that $\gamma_1 \cup \gamma_2$ bounds a
  Whitney disc $D_1$.  We will describe a framing of
  $\nu_{D_1}|_{\partial D_1}$.  Choose a framing of $\nu_{\gamma_1
    \subset \Sigma_1}$, a nonvanishing vector field in the normal
  bundle of $\gamma_1$ in $\Sigma_1$.  This yields a nonvanishing
  vector field in $\nu_{D_1}|_{\gamma_1}$.  Along $\gamma_2$ we choose
  a vector field in $\nu_{D_1}|_{\gamma_2} \cap \nu_{\Sigma_2}$, which
  agrees at $p$ and $q$ with the vector field along $\gamma_1$ which
  we have already chosen (for this to be possible we need that $p$ and
  $q$ are of opposite signs.)  Note that the intersection
  $\nu_{D_1}|_{\gamma_2} \cap \nu_{\Sigma_2}$ is a 1-dimensional
  bundle.  The resulting framing along $\partial D_1 = \gamma_1 \cup
  \gamma_2$ is the \emph{Whitney framing}.  The transport of the
  Whitney framing to $\partial W$ along $D_1$ is depicted in
  Figure~\ref{figure:whitney-framing}.
\end{definition}

\begin{figure}[H]
\begin{center}
\begin{tikzpicture}[scale=.75]
\node[anchor=south west,inner sep=0,scale=.75] at (0,0)
{\includegraphics[scale=0.3]{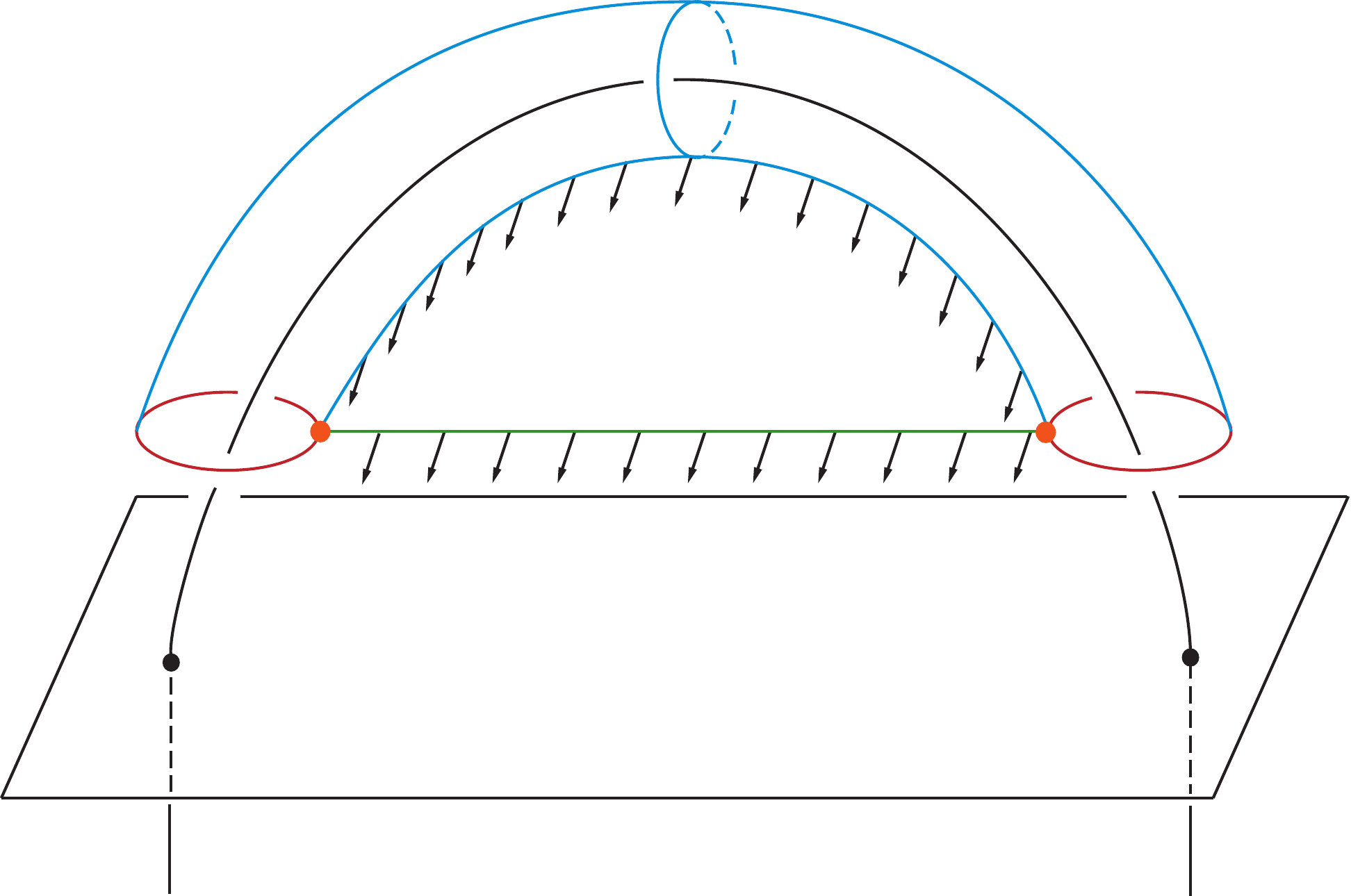}};
\small
\node at (9.7,1.7)  {$\Sigma_1$};
\node at (1.6,.2)  {$\Sigma_2$};
\end{tikzpicture}
\end{center}
\caption{The Whitney framing of the normal bundle of a Whitney disc
  along the boundary. It is tangent to $\Sigma_1$, which appears in
  the picture as a plane, but normal to the $\Sigma_2$, which is the
  surface that appears as a line in the picture. }
    \label{figure:whitney-framing}
\end{figure}

Compare this framing to the disc framing, that is the unique framing
of the normal bundle to~$D_1^i$ restricted to $\partial D_1^i$, in
order to obtain the twisting coefficient $a_i \in \Z$ of~$D_1^i$.
Recall that for the purposes of assigning an integer $a_i$, the disc
framing is considered to be the zero framing.  A Whitney disc is said
to be framed if and only if it has coefficient~$0$; equivalently a
Whitney disc is framed if the Whitney framing and the disc framing
coincide.

As remarked in the introduction, interior twists change the disc
framing by $\pm 2$ relative to the Whitney framing, so we can arrange
that the twisting coefficient is either $1$ or $0$.  Whether or not
this step is performed, the entries of $\lambda$ (and therefore of the
matrix $\Omega$) do not change.

While the Whitney framing defined above is standard (see
\cite[pages~54--8]{Scorpan-2005} for a nice exposition), a framing of
the boundary of an accessory disc does not seem to be standard.
However we will need a detailed understanding of this in order to
compute the matrix of the intersection form of $W$.

\begin{definition}[Accessory framing]\label{defn:accessory-framing}
  Consider the double point loop $\gamma$ of an intersection point $p$
  of $D_0$, which bounds an accessory disc $A_1^i$.  By restricting
  the normal bundle of $D_0$ to $\gamma$, and looking at $W
  \cap \partial (\cl(\nu D_0)|_\gamma)$, we obtain the image $N$ of a
  map into $W$ of an annulus.  Define the curve $\gamma' := A_1^i \cap
  N$.  The boundary $\partial N$ is the union of a longitude and a
  meridian of the Clifford torus $T$ of the double point $p$.  Two
  points $q_1,q_2$ on $\partial N$, one on each component of $\partial
  N \cong S^1 \times S^0$, are identified, where the longitude and
  meridian of the Clifford torus meet.  Thus $\gamma'$ is a simple
  closed curve; in fact $\gamma' = \partial A_1^i$.

  Define the \emph{accessory framing} (or $N$-tangential framing) of
  $A^i_1$ restricted to $\gamma' = \partial A_1^i$ to be a framing of
  the normal bundle of $\gamma'$ by a nonzero vector field in the
  tangent bundle to $N$, except with a slight modification in a
  neighbourhood of $q= q_1=q_2$ that moves the vector field away from
  the tangent bundle~$TN$, as shown in
  Figure~\ref{figure:accessory-framing}; this modification is
  necessary in order for the framing to be well defined at~$q$.

  \begin{figure}[H]
    \begin{center}
      \begin{tikzpicture}[scale=.75]
        \node[anchor=south west,inner sep=0,scale=.75] at (0,0)
        {\includegraphics[scale=0.3]{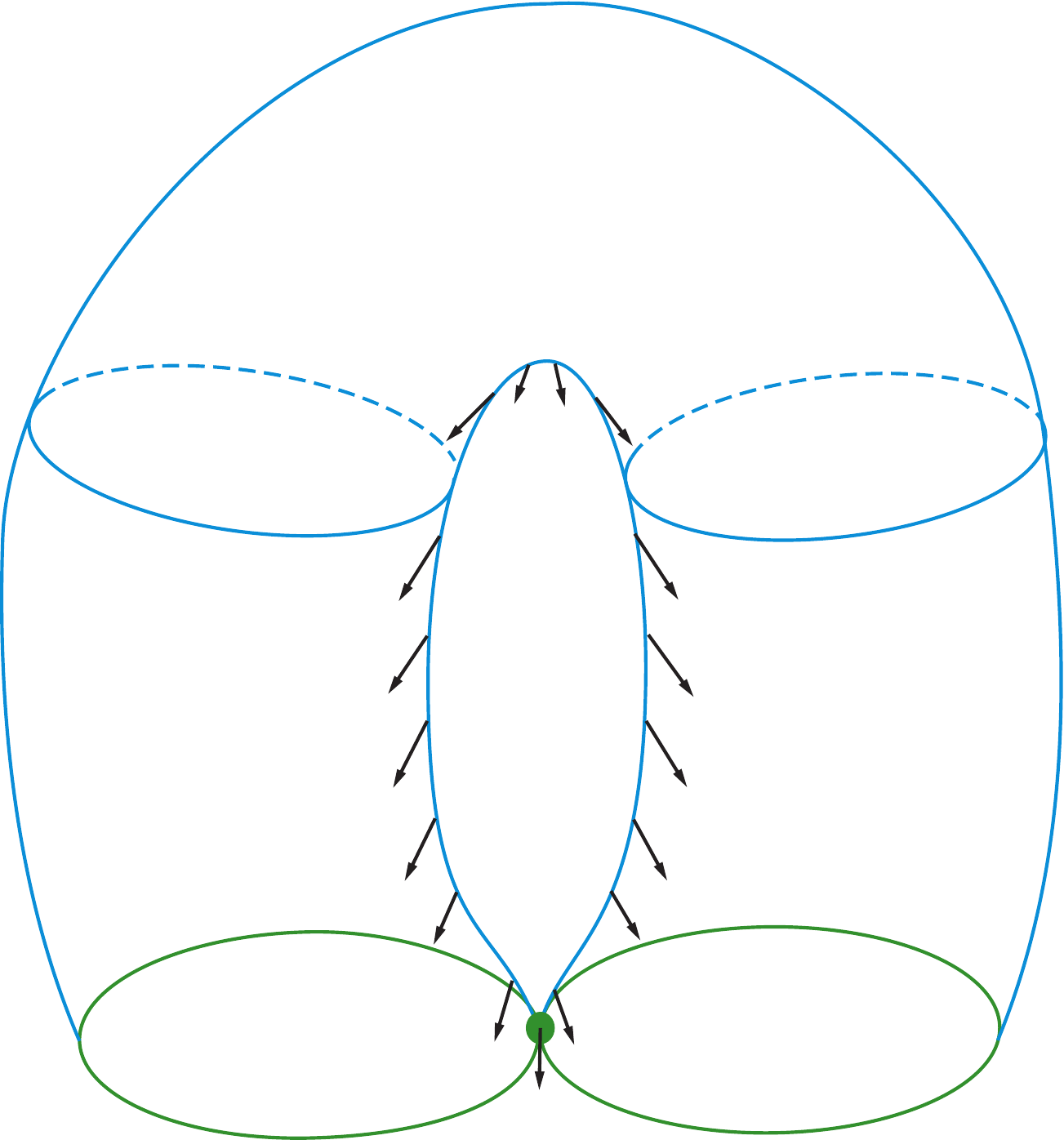}};
      \end{tikzpicture}
    \end{center}
    \caption{The accessory framing.}
    \label{figure:accessory-framing}
  \end{figure}

  More precisely, near an $\R^2\times \R^2$ neighbourhood of $q=0$ in
  which the sheets are $\R^2\times 0$ and $0\times \R^2$, we have
  $T=S^1\times S^1$,
  $N=(S^1\times \R_{\ge 1} \times 0) \cup (0\times\R_{\ge 1}\times
  S^1)$,
  $A_1^i = 0\times\R_{\ge 1} \times \R_{\ge 1} \times 0$, and
  $\gamma' = (0\times 1\times \R_{\ge 1} \times 0) \cup
  (0\times\R_{\ge 1} \times 1\times 0)$.
  The framing is $(1,0,0,0)$ on $0\times 1\times \R_{\ge 2} \times 0$,
  is $(0,0,0,1)$ on $0\times\R_{\ge 2} \times 1\times 0$, and is of
  the form $(\cos \theta,0,0,\sin \theta)$ with $0\le\theta\le\pi/2$
  on the remaining part.  Specifically,
  \[(\cos(\pi t/4),0,0,\sin (\pi t/4)) \text{ on } \begin{cases}
    (0,1,2-t,0) & t \in [0,1] \\
    (0,t,1,0) & t \in [1,2].
  \end{cases}\]
  We remark that this is equal to the framing used by
  Casson \cite{Casson-1986-towers} for the construction of a Casson
  handle.
\end{definition}

Compare the accessory framing to the disc framing of $A_1^i$, in order
to compute the twisting coefficient $b_i \in \Z$ which occurs in the
diagonal terms $\lambda_{2i}$ of $\lambda$. Recall that for the
purposes of assigning an integer $b_i$, the disc framing is considered
to be the zero framing.

\subsection{Equivariant intersections of the spheres $S_i$}
\label{section:intersection-of-spheres-0}

Sections~\ref{section:intersection-of-spheres-0}--\ref{section:int-spheres-4}
describe the intersections amongst the spheres $S_i$.   Together these
sections prove~(\ref{item:2-int-form-thm}),
(\ref{item:4-int-form-thm}) and~(\ref{item:5-int-form-thm}) of
Theorem~\ref{theorem:intersection-form-using-whitney-discs}.  Only the
computations of Sections \ref{section:intersection-of-spheres-0},
\ref{section:framing-conditions-for-surgery}~and
\ref{section:int-spheres-4} are required for the proof of
~(\ref{item:2-int-form-thm}), (\ref{item:4-int-form-thm})
and~(\ref{item:5-int-form-thm}) of
Theorem~\ref{theorem:intersection-form-using-accessory-discs}.

We begin with a lemma translating intersections with a Whitney or
accessory disc into the intersection numbers from the intersections
with a sphere $S_i$.

In the next lemma let $\Sigma$ be a surface in $W$ with a path from a
basepoint of $\Sigma$ to the basepoint of $W$, for which
$\pi_1(\Sigma) \to \pi_1(W)$ is the trivial map.  Recall that $z :=
(1-t)(1-t^{-1})$.

\begin{lemma}\label{lemma:4-parallel-copies}
   For each intersection point of $\Sigma
  \cap D_1^j$ (respectively $\Sigma \cap A_1^j$), there are four
  resulting intersections of $\Sigma \cap S_{2j-1}$ $($respectively
  $\Sigma \cap S_{2j}$$)$.  If the $\Z[\Z]$ intersection number of the
  intersections of $\Sigma$ with $D_1^j$ $($respectively $A_1^j$$)$ is
  $p(t)$, then the $\Z[\Z]$ intersection number with $S_{2j-1}$
  (respectively $S_{2j}$) is $z\cdot p(t)$.
\end{lemma}

\begin{proof}
  We discuss the case of Whitney discs and odd indexed spheres first.
  Assume that there is a single intersection point in $\Sigma \cap
  D_1^j$ and it has $\Z[\Z]$ intersection number~$+1$.

  Consider the four copies of the Whitney disc $D_1^j$ which occur in
  $S_{2j-1}$.  First we use two
  copies of $D_1^j$ to surger an annulus $N_2$ into a disc $C$.  These
  copies of $D_1^j$ are called $(D_1^j)_{\pm}$.  Label so that going
  from $(D_1^j)_+$ to $(D_1^j)_-$ along $N_2$ involves traversing a
  meridian of $D_0$ in the positive sense.

  Then we use two copies $C_{\pm}$ of $C$ to surger the torus
  $T_{12}$.  Label so that going from $C_+$ to $C_-$ along $T_{12}$
  involves traversing a meridian of $D_0$ in the negative sense.
  Creating $C_+$ and $C_-$ requires two copies of each of
  $(D_1^j)_{\pm}$, which we call $(D_1^j)_{\pm \pm}$.   Observe that $C_+$ uses $(D_1^j)_{++}$ and
  $(D_1^j)_{-+}$, while $C_-$ uses $(D_1^j)_{+-}$ and $(D_1^j)_{--}$.
  If $\Sigma$ intersects $D_1^j$ in a point then $\Sigma$ intersects
  each of the $(D_1^j)_{\pm \pm}$ in a point.

  In order for $S_{2j-1}$ to be oriented, we need to take the opposite
  orientations on $(D_1^j)_{+-}$ and~$(D_1^j)_{-+}$.  Choose the
  orientation of $S_{2j-1}$ to be such that the intersection signs for
  $\Sigma \cap (D_1^j)_{\zeta\xi}$ is equal to $\zeta\cdot\xi$ for
  $\zeta,\xi \in \{+,-\}$.

  We are given a choice of path from the basepoint of $W$ to the
  basepoint of $D_1^j$. Use the same path, perturbed slightly, with
  the basepoint of $S_{2i-1}$ located on $D^j_{++}$.  With respect to
  this choice of basepoint, the contributions from the intersections
  of $\Sigma$ with $(D_1^j)_{++}, (D_1^j)_{+-}, (D_1^j)_{--},
  (D_1^j)_{-+}$ are $+1$, $-t^{-1}$, $+1$, $-t$ respectively.  The sum
  is $2-t-t^{-1} = (1-t)(1-t^{-1})=z$.  See
  Figure~\ref{figure:equivariant-intersection}.

\begin{figure}[t]
\begin{center}
\begin{tikzpicture}[scale=.75]
\node[anchor=south west,inner sep=0,scale=.75] at (0,0)
{\includegraphics[scale=0.3]{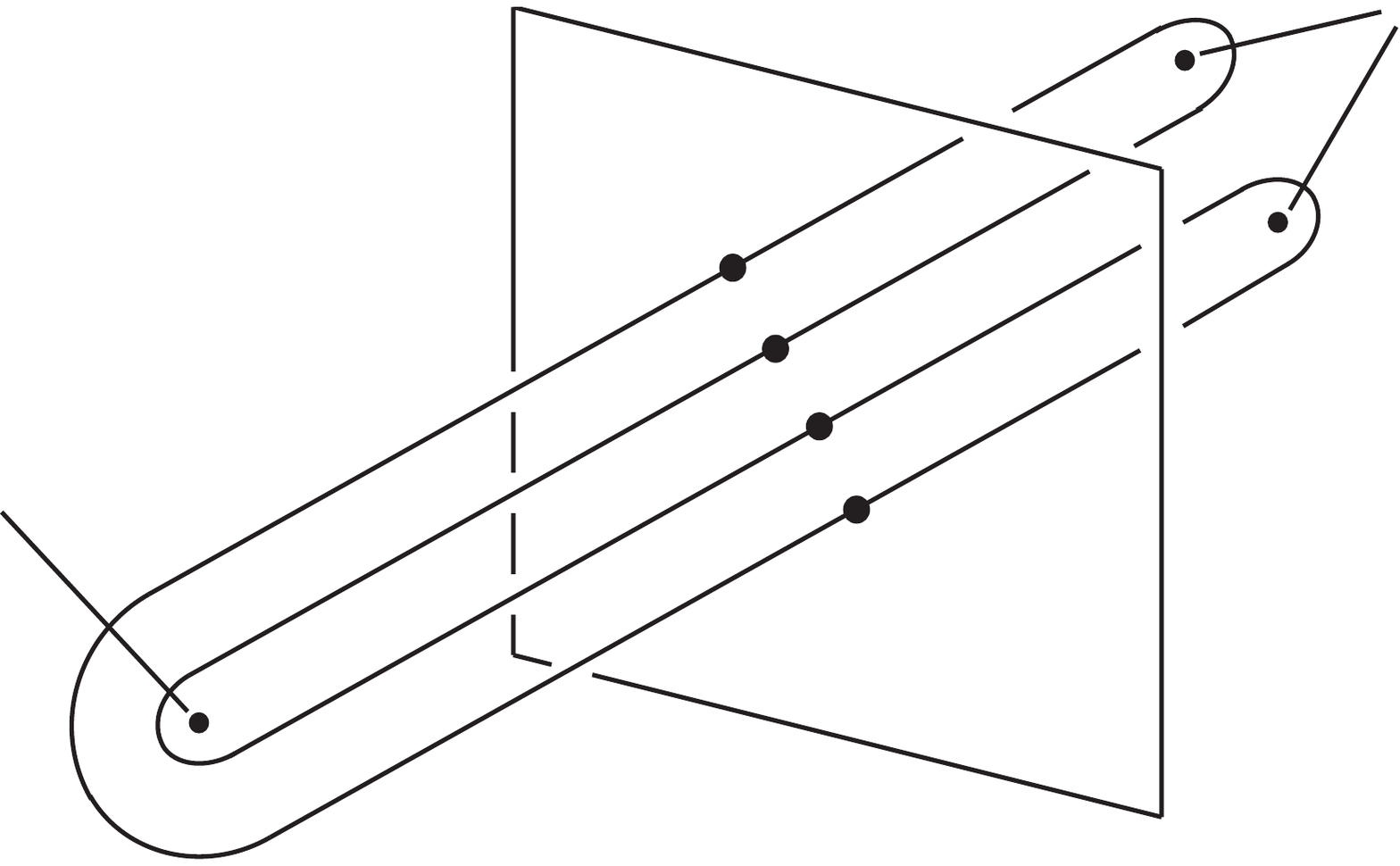}};
\small
\node at (7.2,2)  {$\Sigma$};
\node at (8.55,5.85)  {$D_0$};
\node at (-.27,2.85)  {$D_0$};
\draw[<-]  (1.3,2.6)-- ++(-.7,.7) node[above] {$(D_1^j)_{++}$};
\draw[<-]  (2.7,2.8)-- ++(-1,1) node[above] {$(D_1^j)_{+-}$};
\draw[<-]  (2.8,2)-- ++(0,-1.3) node[below] {$(D_1^j)_{--}$};
\draw[<-]  (1.9,.9)-- ++(0,-.7) node[below] {$(D_1^j)_{-+}$};
\end{tikzpicture}
\end{center}
\caption{ A schematic diagram of four intersection points of a sphere
  $S_{2i-1}$ (or $S_{2i}$) with a surface $\Sigma$ arising from a
  single intersection point of $D_{i}^1$ (or $A^1_i$) with $\Sigma$.}
    \label{figure:equivariant-intersection}
\end{figure}

Add the contributions from multiple intersection points in $D_1^j \cap
\Sigma$ to obtain the desired result.  If the initial $\Z[\Z]$
intersection number of a point of intersection between $D^j_1$ and
$\Sigma$ is $\pm t^{\ell}$, then the contribution to the intersection
number of $\Sigma$ with $S_{2j-1}$ is~$\pm zt^{\ell}$.

The result for the intersection number of the sphere $S_{2j}$ with a
surface $\Sigma$ in terms of the intersection number of $\Sigma$ with
$A_1^j$ is proved in the same way, with $A_1^j$ replacing $D_1^j$, $N$
replacing $N_2$, $D$ replacing $C$, and with $T$ replacing~$T_{12}$.
\end{proof}

\subsection{Intersection of $S_{i}$ with $S_{j}$ for $i \neq j$ and $\{i,j\} \neq \{2i-1,2i\}$}\label{section:int-spheres-1}

First we consider the intersections between the spheres $S_{2i-1}$ and
$S_{2j-1}$ for $i \neq j$.  The $\Z[\Z]$ intersections between the
spheres $S_{2i-1}$ and $S_{2j-1}$ for $i \neq j$ arise directly from
intersections between the order one Whitney discs $D_1^i$ and~$D_1^j$.

We investigate the contribution of a single intersection point between
$D_1^i$ and $D_1^j$ with associated element~$\pm t^{\ell}$.  Since
$S_{2i-1}$ contains 4 parallel copies of $D_1^i$ and $S_{2j-1}$
contains $4$ parallel copies of $D_1^j$, there are 16 intersection
points in $S_{2i-1} \cap S_{2j-1}$ arising from the single
intersection point in $D_1^i \cap D_1^j$.

We apply Lemma~\ref{lemma:4-parallel-copies} five times, once with
$\Sigma=D_1^i$ and $D_1^j$ as the intersecting disc, and then once
with $\Sigma$ as each of the four parallel copies of $D_1^j$ in
$S_{2j-1}$, and $D_1^i$ the intersecting disc.  The resulting
$\Z[\Z]$-intersection number is therefore~$\pm z^2 t^{\ell}$.

The intersections of $S_{2i-1}$ with $S_{2j}$ for $i \neq j$ and the
intersection of $S_{2i}$ with $S_{2j}$ for $i \neq j$ are computed in
the same way, except that a sphere with even index $S_{2i}$ contains four parallel copies of an accessory disc instead of a Whitney disc.

\subsection{Intersection of $S_{2i-1}$ and $S_{2i}$}\label{section:int-spheres-2}

During the construction of $S_{2i-1}$ and $S_{2i}$ we must be careful to make sure that the intersections are transverse.  There is one Clifford torus associated to one of the double points paired up by $D_1^i$, say $T_2$, a parallel copy of which is also used as the Clifford torus $T$ to surger using $A_1^i$ in the construction of $S_{2i}$.   We may assume that $T_2$ and $T$ are associated to a self-intersection point of $D_0$ of positive sign.  We use a slightly bigger Clifford torus for $T_2$ than for $T$.   As a result $T$ is disjoint from $S_{2i-1}$ but $T_2$ intersects $A_1^i$ in a single point.  Apply Lemma~\ref{lemma:4-parallel-copies} to obtain a contribution of $z$ to the off-diagonal entries of each $2 \times 2$ block of the matrix $X$ from Theorem~\ref{theorem:intersection-form-using-whitney-discs}.

\subsection{Framing conditions for surgery}\label{section:framing-conditions-for-surgery}

To understand the self intersection terms, first we need to give a
description of the framing conditions that must hold in order for
surgery to be performed and the normal bundle of the outcome to again
be framed.  One can still perform surgery without the framing
condition, but then it becomes tricky to verify that one is keeping
track of intersection numbers and framing conditions correctly.

Recall that a framing of a surface in a 4-manifold means a framing of
its 2-dimensional normal bundle, and a framing is specified by a
single nonvanishing vector field in the normal bundle.  A second
nonvanishing vector field can then be found using the orientation of
the normal bundle, which is itself inherited from the orientation of
the surface and the orientation of the ambient 4-manifold.

Let $V$ be a 4-manifold, let $T \subset V$ be an embedded torus with
trivial normal bundle, with an essential, simple closed curve $\gamma
\subset T$, and let $D\looparrowright V$ be an immersed disc which we
want to use to perform surgery on $T$, so that $\partial D = \gamma$.

  There is a unique framing of $D$ in $W$, that is,
trivialisation of the normal bundle~$\nu_D$,
which we call the \emph{disc framing}.  In addition, suppose we have
the following data:
\begin{itemize}
\item A framing of $T$ in $W$, which we call a \emph{surgery framing}.
\item A framing $f_{\gamma\subset T}$ of $\gamma\subset T$, that is, a trivialisation of the
  normal bundle~$\nu_{\gamma \subset T}$.
\item A framing $f_{\gamma\subset D}$ of $\gamma\subset D$.
\end{itemize}
The various vector bundles on $\gamma$ are shown in
Figure~\ref{figure:framing-conditions}.  Note that the framings of
$\gamma \subset T$ and $\gamma \subset D$ are uniquely determined up
to negation, while that of $T\subset W$ is not.

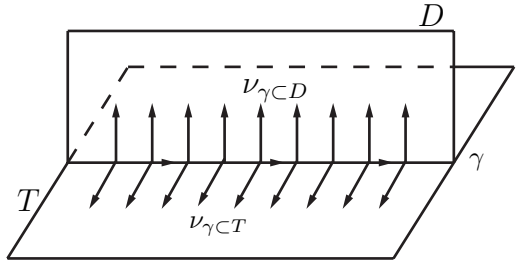
\begin{figure}[H]
\begin{center}
\begin{tikzpicture}
\node[anchor=south west,inner sep=0] at (0,0){\includegraphics[scale=0.5]{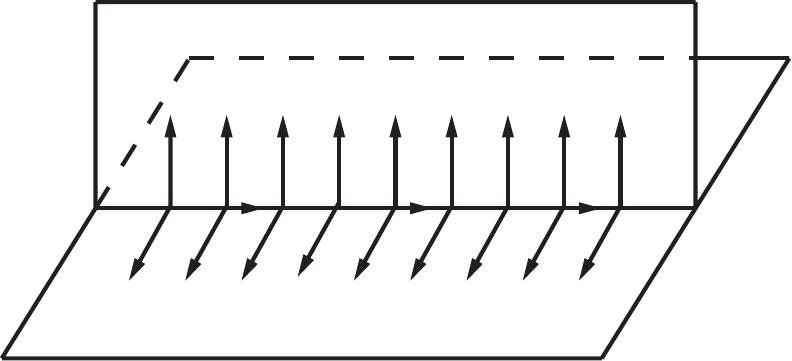}};
    \node at (5.6, 3.25)  {\larger $D$};
    \node at (3.55,2.25)  {\larger $\nu_{\gamma \subset D}$};
    \node at (2.8,0.45)  {$\nu_{\gamma \subset T}$};
    \node at (6.2,1.3)  {$\gamma$};
    \node at (0.3,0.8)  {\larger $T$};
\end{tikzpicture}
\end{center}
\caption{A 3-dimensional slice of a neighbourhood of a point of
  $\gamma$, with the surfaces $T$ and $D$ shown, together with
  trivialisations of the tangent bundle $T\gamma$ and of the normal
  bundles $\nu_{\gamma \subset D}$ and $\nu_{\gamma \subset T}$. The
  vector field $\mbf{w}$ is in the 4th dimension and so is not visible
  in the picture.}
    \label{figure:framing-conditions}
\end{figure}

In order for the surgery to yield a \emph{framed}
2-sphere, we require the following: there exists a vector field
$\mbf{w}$ on $\gamma$ such that
\begin{enumerate}
\item[(F1)] $(f_{\gamma\subset T},\mbf{w})$ is equivalent to the disc
  framing on~$\gamma$.
\item[(F2)] $(f_{\gamma\subset D},\mbf{w})$ is equivalent to the
  surgery framing on~$\gamma$.
\end{enumerate}
In order for the conditions (F1) and (F2) to hold we might have to
make some modifications of the original data.  First we may need to
boundary twist $D$ around $\gamma$, introducing one intersection in $D
\cap T$ for each twist, until there exists a $\mbf{w}$
satisfying~(F1).  Since it is constrained to a single dimension, up to
homotopy $\mbf{w}$ is determined up to sign, and the sign is
determined by the other choices of framing.  Since $\gamma$ is
essential, we are then free to change the surgery framing of $T$ along
$\gamma$, until (F2) holds.  In the sequel this will always be done
without further comment.

We may then use the surgery framing to take two parallel copies of $D$
and construct a framed sphere $S$.  The framing on $S$ is obtained by
taking the framing on $\nu D$ on one copy of $D$, its negative on the
other copy of $D$, the framing of $\nu T$ on $T \sm (\gamma \times
D^1)$, and then smoothing the corners by rotating between the two
vector fields in a neighbourhood of $\gamma \times \{\pm 1\}$.  The
rotation occurs in the 2-dimensional subbundle of $TV|_{\gamma}$ which
is orthogonal to $T\gamma$ and $\mbf{w}$.

\subsection{Self intersection of $S_{2i-1}$}\label{section:int-spheres-3}

First, we note that each self-intersection of the disc $D^i_1$ gives
rise to 16 self intersection points of $S_{2i-1}$, which means that we
should count 32 intersection points between $S_{2i-1}$ and a push-off.

Given a self-intersection point $p$ of $D_1^i$ with double point loop
$t^{\ell}$ and sign $\pm$, the intersection number between $D_1^i$ and
a parallel push-off is $\pm(t^\ell +t^{-\ell})$.  We can only define
the double point loop up to the indeterminacy $t^{\ell}=t^{-\ell}$,
since we have no canonical ordering of sheets.  Of course $t^\ell
+t^{-\ell}$ is independent of the choice here.  Now apply the argument
of Section~\ref{section:int-spheres-1} to yield a coefficient of
$z^2$, noting that $z=\ol{z}$.  This accounts for the diagonal terms
of $z(zY + \ol{zY}^T)$.  There are indeed 32 terms for each $\pm
t^{\ell}$ summand of $Y$.

The potential twisting of the Whitney discs gives the crucial extra
terms.  We want the sphere $S_{2i-1}$ to be framed, in order to be
able to compute the self intersection number
$\lambda(S_{2i-1},S_{2i-1})$ by counting intersection points between
$S_{2i-1}$ and a parallel push-off.  The twisting occurs in the first
step, during the construction of $C$ from $N_2$ and $(D_1^i)_{\pm}$.

Recall that we denote $\a_1 = N_1 \cap \partial D_1^i$ and $\a_2 = N_2
\cap \partial D_1^i$.  The notation $\a_1,\a_2$ was also used for the
Whitney arcs which lie on $D_0$, so we make a slight abuse to use the
same notation for their push-offs onto $N_1$, $N_2$ respectively.

Align the disc framing of $D_1^i$ with the Whitney framing along
$\a_1$.  Note that, within the homotopy class, we are free to adjust
any framing on an interval.  Then look at the disc framing of $D_1^i$
restricted to $\a_2$.  The difference between this framing and Whitney
framing, which is also the surgery framing along $N_2$, is the
twisting coefficient $a_i$.  Introduce $a_i$ boundary twists along
$\a_2$.  Twisting is described
in~\cite[Section~1.3]{Freedman-Quinn:1990-1}.  (With respect to the
whole of the Whitney disc, as originally pairing intersections of
$D_0$, this is an interior twist. However with respect to the sub-disc
whose boundary is $(N_1 \cap \partial D_1^i) \cup (N_2 \cap \partial
D_1^i)$, this is a boundary twist.  Only the part of the Whitney disc
that we use for surgery is relevant.)  The boundary twist changes the
Whitney disc, and therefore the disc framing, so that it now coincides
with the surgery framing along $N_2$.  Strictly speaking, for these
boundary twists, we should push $N_2$ slightly off $\partial W$.

The Whitney framing along $N_1$ differs from the surgery framing on
$T_{12}$ by a fixed rotation.  Both are normal to $D_1^i$ along $N_1
\cap \partial D_1^i$.  Therefore in a neighbourhood of $\a_1$ we can
arrange the disc framing by a homotopy so that it lies in $\nu_{\a_1
  \subset N_1}$.

The disc framing of $C$ is constructed from the disc framing of
$(D_1^i)_{+}$, the negative of the disc framing of $(D_1^i)_-$ and the
normal framing to $N_2$.  This latter is also the disc framing of $N_2
\sm (\a_2 \times D^1)$.  The fact that we obtain the disc framing of
$C$ is guaranteed by the boundary twists above.  For the second
surgery, converting $T_{12}$ to $S_{2i-1}$ using $C_{\pm}$, the
framings already coincide as required by
Section~\ref{section:framing-conditions-for-surgery}.  Therefore no
more boundary twisting is required.

Now we consider the contribution of a boundary twist as above to the
self intersection number.  Each boundary twist produces a single
intersection point between $N_2$ and $D_1^i$.  It therefore produces
two self-intersection points of $C$.

Two copies of $N_2$ will be in the final sphere $S_{2i-1}$.  To
compute the self intersection number $\lambda(S_{2i-1}, S_{2i-1})$,
first we compute the Wall self intersection
$\mu(S_{2i-1})$~\cite[Chapter~5]{Wall-Ranicki:1999-1}, and observe
that $\lambda(S_{2i-1}, S_{2i-1}) = \mu(S_{2i-1}) +
\ol{\mu(S_{2i-1})}$.  This works for two reasons.  First, the sphere
$S_{2i-1}$ is framed, as we just went to great lengths above to
ensure.  Thus there is no extra term from the Euler characteristic of
the normal bundle~\cite[Theorem~5.2~(iii)]{Wall-Ranicki:1999-1}.
Second, although the self-intersection $\mu(S_{2i-1})$ is only
well-defined up to the indeterminacy $a= \ol{a}$, the sum
$\mu(S_{2i-1}) + \ol{\mu(S_{2i-1})}$ is well-defined and determines a
unique element of $\Z[\Z]$.

Label the two copies of $N_2$ which occur in $C_{\pm}$ by
$(N_2)_{\pm}$.  The intersection numbers of these with $D_1^j$ are $1$
and $-t$ respectively, since the two intersections differ by a
meridian of $D_0$.  By Lemma~\ref{lemma:4-parallel-copies}, the
contribution to the self intersection number from each boundary twist
is therefore $(1-t)z$.  Therefore the contribution to
$\lambda(S_{2i-1}, S_{2i-1})$ is
\[
 (1-t)z +
\ol{(1-t)z} = (1-t+1-t^{-1})z = z^2.
\]
All together the boundary twists therefore
contribute $a_i z^2$ to $\lambda(S_{2i-1}, S_{2i-1})$.

\subsection{Self intersection of $S_{2i}$}\label{section:int-spheres-4}

There are three types of contributions to the self intersection of
$S_{2i}$.  First, a self-intersection of the disc $A_1^i$ with
$\Z[\Z]$-intersection number $p(t)$ contributes $z(zp(t) +
\ol{zp(t)})$, by the analogous argument as for the spheres $S_{2i-1}$
in Section~\ref{section:int-spheres-3}.

The twisting $b_i$ of the accessory framing
(Figure~\ref{figure:accessory-framing}) with respect to the disc
framing contributes $b_iz^2$, by a similar argument to that in
Section~\ref{section:int-spheres-3}.  We give the outline.  Again we
need that the disc framing of $D$ is constructed from the disc
framings of $N \sm (D^1 \times D^1)$ and $A_+$ together with the
negative of the disc framing of $A_-$.  To achieve this perform $b_i$
boundary twists of $A_1^i$ around $A_1^i \cap N$.  These contribute
$b_iz^2$ to $\lambda(S_{2i},S_{2i})$ as claimed.

In the construction of the spheres $S_{2i-1}$, the first set of
boundary twists was sufficient: after this the second surgery, of
$T_{12}$ into a sphere, was automatically correctly framed.  However,
for the spheres $S_{2i}$ constructed from the accessory discs, that we
consider in this section, this is not the case.

\begin{claim*}
  The surgery framing of the $(1,1)$ curve on the Clifford torus $T$
  is $+1$ with respect to the disc framing on $D$.
\end{claim*}

Given the claim, we perform a single boundary twist of $D$ about its
boundary, before using it to surgery $T$ into $S_{2i}$.  This gives
rise to a contribution of $1-t$ to the self intersection
$\mu(S_{2i})$, therefore a contribution of $\mu(S_{2i}) +
\ol{\mu(S_{2i})} = 1-t + 1-t^{-1} = z$ to $\lambda(S_{2i},S_{2i})$ as
desired.

Roughly, the $+1$ from the claim arises from the self linking of the
$(1,1)$ curve on the Clifford torus.  This was previously observed in
a different context in~\cite[Lemma~4]{Freedman-Kirby:1978}.  Note that
if the sign of the associated double point of $D_0$ were $-1$, then
the difference in framings would instead be $-1$.

The claim will follow from the observation of the next lemma.  In
order to state the lemma, we describe a disc $D'$ in a~$D^4$
neighbourhood of a double point $p$ of $D_0$, whose boundary is the
$(1,1)$-curve on the Clifford torus $T$ i.e.\ the boundaries of $D$
and $D'$ coincide.  Recall that the Clifford torus is $T = S^1 \times
S^1 \subset \R^2 \times \R^2 \cong \R^4 \cong D^4$.  The meridian is
$S^1 \times \{1\}$ and the longitude is $\{1\} \times S^1$.  Take the
union of the two discs $D^2 \times \{1\}$ and $\{1\} \times D^2$ and
add two small triangles as shown in
Figure~\ref{figure:framing-small-triangles}.

\begin{figure}[t]
\begin{center}
\begin{tikzpicture}[scale=.8]
\node[anchor=south west,inner sep=0,scale=.8] at (0,0)
{\includegraphics[scale=0.23]{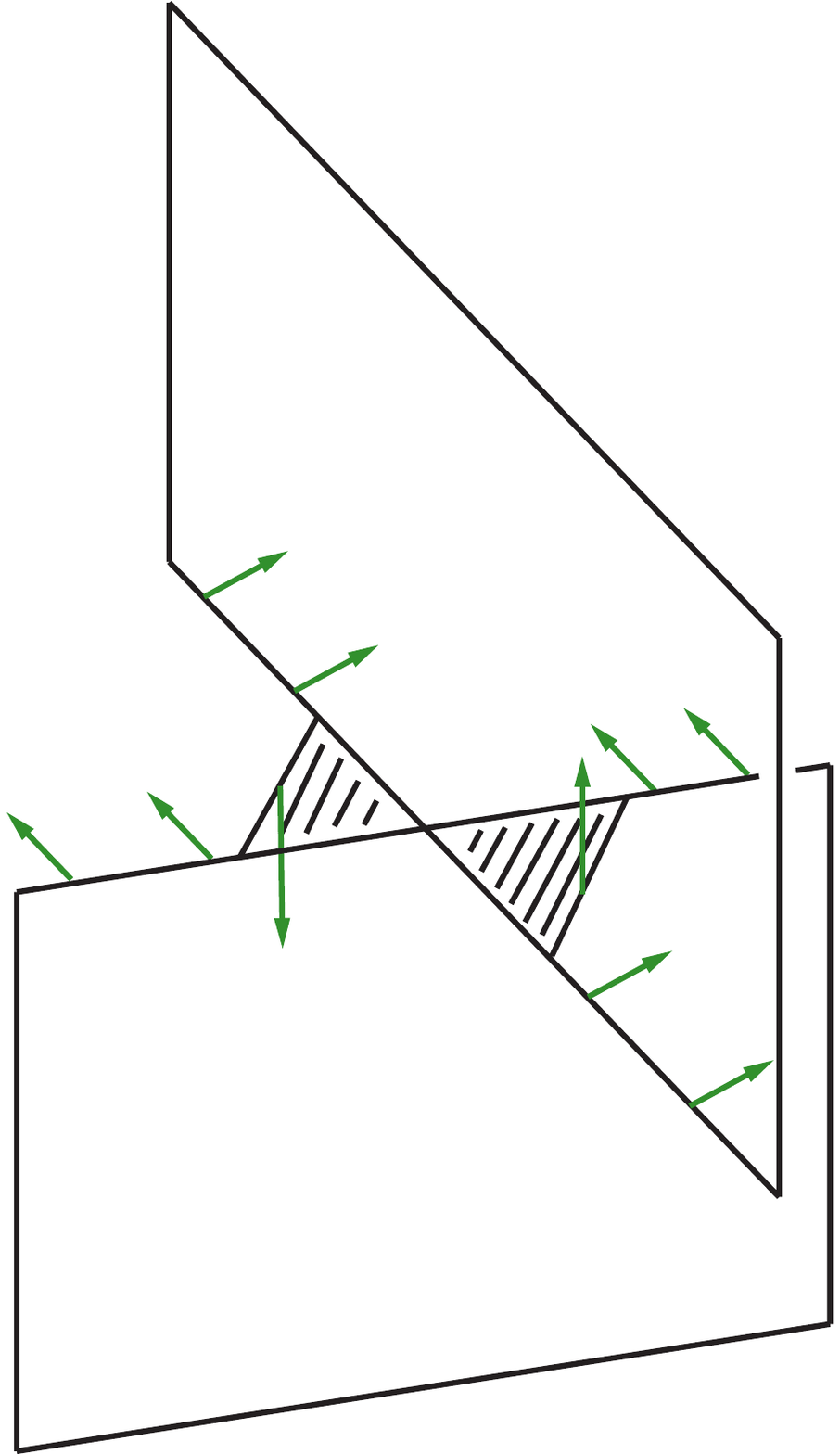}};
\node at (4.6,2)  {\small $D'$};
\end{tikzpicture}
\end{center}
\caption{The framing of the normal bundle of $D'$ restricted to
  $\partial D'$, in a neighbourhood of the intersection on $T$ of the
  meridian and longitude.  The shaded triangles are part of $D'$, and
  the fact that the framing stays normal to them means that it twists
  out of the tangent bundle of $T$.}
    \label{figure:framing-small-triangles}
\end{figure}

\begin{lemma}\label{lemma:disc-for-framing-computation}
  There exists a 3-ball $B$ in $D^4$ whose boundary is the 2-sphere
  formed from the union of the surgery disc $D$ with the disc $D'$.
  Moreover there exists a framing for the normal bundle of $B$ which
  restricts to the disc framings of both $D$ and $D'$.
\end{lemma}

\begin{proof}
  The 3-ball $B$ is constructed from glueing together $A_1^i \times
  [-1,1]$ and $\a \times D^2$ --- recall that $A_1^i \times \{\pm 1\}
  \cong A_{\pm}$ and $\a \times S^1 =N$.  The normal bundle of $B$ is
  one dimensional, so the framing only depends on a choice of sign.
  The framing determines a nonvanishing vector field in the normal
  bundle of $D$ and $D'$, which therefore must restrict to the disc
  framings on their common normal boundary.
\end{proof}

By Lemma~\ref{lemma:disc-for-framing-computation}, we can compute the
disc framing of $D$ restricted to its boundary by computing the disc
framing of $D'$.  The surgery framing is $+1$ with respect to the disc
framing of $D'$. The surgery framing is shown in
Figure~\ref{figure:surgery-framing-1-1-curve}, where we see that the
linking number of the two curves is $+1$.

\begin{figure}[H]
\begin{center}
\begin{tikzpicture}[scale=.7]
\node[anchor=south west,inner sep=0,scale=.7] at (0,0)
{\includegraphics[scale=0.4]{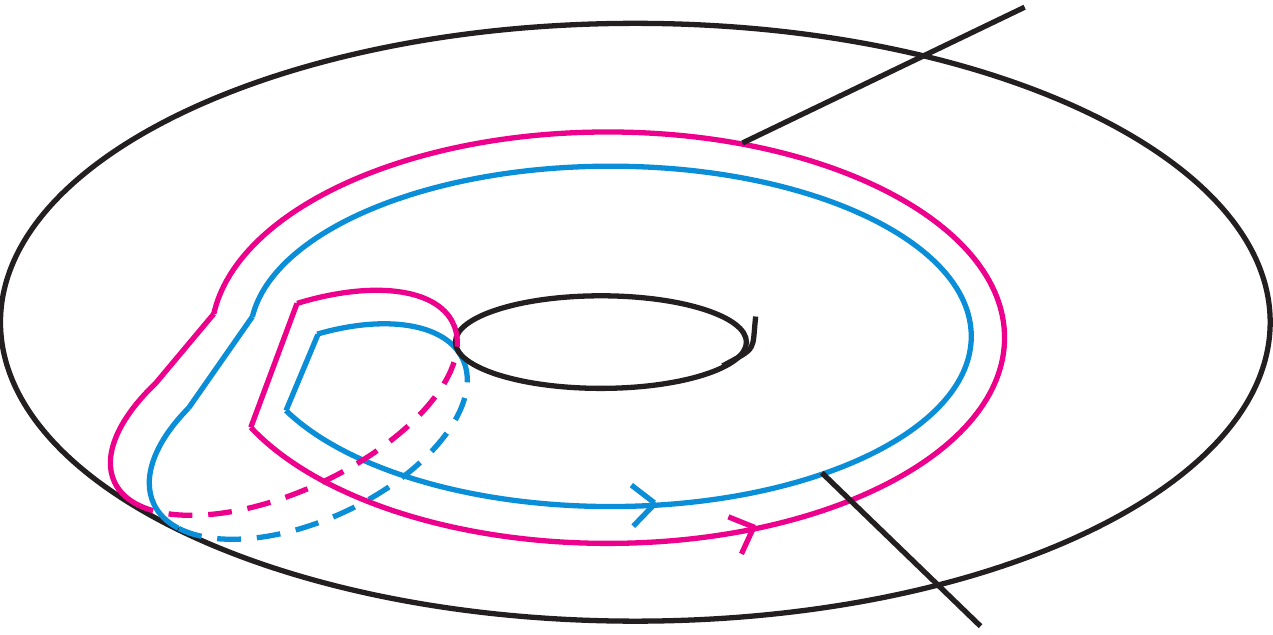}};
\small
\node at (8.9,2.1)  {$T$};
\node at (7.4,-.3)  {$(1,1)$ curve};
\node at (7.9,4.5)  {surgery push-off};
\end{tikzpicture}
\end{center}
\caption{The $(1,1)$-curve on the Clifford torus $T$ and a push-off
  using the surgery framing.  The linking number in $S^3$ is $+1$.}
    \label{figure:surgery-framing-1-1-curve}
\end{figure}
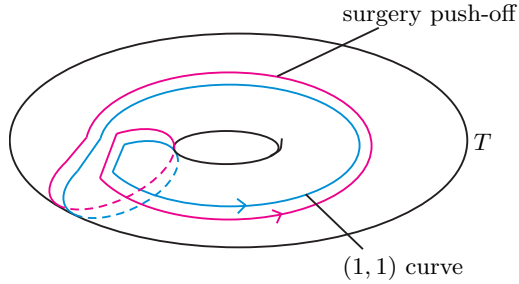

To compute the framing of $D'$, isotope it in a collar neighbourhood
of the boundary so that a (smaller) collar neighbourhood lies in
$S^3$.  The framing of $D'$ and the surgery framing agree along the
meridian of $T$, are opposite along the longitude, and in a
neighbourhood of the intersection point of the longitude and the
meridian of $T$ there is a rotation.  The arrangement is as shown in
Figure~\ref{figure:framing-small-triangles}.  As the framing vector
for $D'$ stays normal to the two small triangles we see that it
undertakes a single full $-1$ twist with respect to the surgery
framing.  We compute that the framing of $D'$ induces a push-off which
has linking number zero with the $(1,1)$ curve of~$T$.  Thus the
surgery framing is $+1$ with respect to the disc framing.  This
completes the proof of the claim and therefore of the computation of
the self-intersection of the spheres $S_{2i}$.


\section{Homology of the boundary of $W$}
\label{section:homology-boundary}

\begin{proposition}\label{prop:order-homology-partial-W-new}
  The first homology $H_1(\partial W;\Z[\Z])$ is isomorphic to
  $H_1(X_K;\Z[\Z])\oplus (\Z[\Z]/\langle z\rangle)^d$.  Consequently
  the order of $H_1(\partial W;\Z[\Z])$ is $(t-1)^{2d}\Delta_K(t)$.
\end{proposition}

\begin{proof}
    As before, let $\nu D_0$ be a (closed) regular neighbourhood of the order zero disc $D_0$
  in~$D^4$.  Since $D_0$ has $d$ double points, $\nu D_0$ is obtained
  by $d$ self plumbings performed on a 2-handle $D^2\times D^2$.  We
  have $W=\cl(D^4\setminus \nu D_0)$.  Let $\partial_+ = \partial(\nu
  D_0) \cap W$ and $\partial_- = \cl(\partial (\nu D_0)
    \setminus \partial_+)$.  Let $X_K=\cl(S^3\setminus \partial_-)$ be
  the exterior of the knot~$K$.  Then $\partial W = \partial_+ \cup
  X_K$ and $\partial_+ \cap X_K=\partial X_K$.

  \begin{figure}[ht]
    \begin{tikzpicture}[
      scale=.4,line width=1pt,
      over/.style={draw=white,double=black,double distance=1pt,line width=1.5pt}
      ]
      \small
      \foreach \y in {-3, 3} {
        \begin{scope}[yshift=\y cm]
          \draw[over,double distance=.6pt] (2,2.5)--(2,0) (2.4,2.5)--(2.4,0);
          \draw[over] (0,0) ..controls +(0,1.2) and +(-1.2,0).. ++(2,2) -- ++(1,0)
          ..controls +(.25,0) and +(0,.25).. ++(.5,-.5) coordinate(a);
          \draw[over] (a) ++(-.5,0) ..controls +(0,.25) and +(-.25,0).. ++(.5,.5)
          -- ++(1,0) ..controls +(1.2,0) and +(0,1.2).. ++(2,-2) node[right] {$0$}
          ..controls +(0,-1.2) and +(1.2,0).. ++(-2,-2)
          -- ++(-2.5,0) ..controls +(-1.2,0) and +(0,-1.2).. ++(-2,2);
          \draw[over] (a) ++(-.5,0) ..controls +(0,-.25) and +(-.25,0).. ++(.5,-.5)
          -- ++(1,0) ..controls +(.6,0) and +(0,.6).. ++(1,-1) ..controls +(0,-.6) and +(.6,0).. ++(-1,-1)
          -- ++(-2.5,0) ..controls +(-.6,0) and +(0,-.6).. ++(-1,1) ..controls +(0,.6) and +(-.6,0).. ++(1,1)
          -- ++(1,0) coordinate(last);
          \draw[over] (last) ..controls +(.25,0) and +(0,-.25).. ++(.5,.5);
          \draw[over,double distance=.6pt] (2,0)--(2,-2.5) (2.4,0)--(2.4,-2.5);
        \end{scope}
      }
      \foreach \y in {0,1,2} \draw[cap=round,line width=1pt] (2,-.3+.3*\y)--+(0,0) (2.4,-.3+.3*\y)--+(0,0);
      \draw[rounded corners=.3cm,line width=.6pt] (2,-5.5) -- ++(0,-1) -- ++(-4,0) -- ++(0,13) -- ++(4,0) -- ++(0,-1);
      \draw[rounded corners=.45cm,line width=.6pt] (2.4,-5.5) -- ++(0,-1.4) -- ++(-4.8,0) -- ++(0,13.8)
      -- ++(4.8,0) -- node[right,xshift=-1mm,yshift=2mm]{$\partial_-$} ++(0,-1.4);
      \foreach \y in {-6.5,6.9} {
        \draw[thin] (0,\y) arc[x radius=.15,y radius=.2,start angle=90,end angle=-90];
        \draw[thin,cap=round,dash pattern=on 0mm off .35mm]
        (0,\y) arc[x radius=.15,y radius=.2,start angle=90,end angle=270];
      }

      \draw[<-,very thick] (7.5,0) -- node[above]{\begin{tabular}{c}covering\\map\end{tabular}} +(4,0) ;

      \begin{scope}[xshift=12cm]
        \draw[over,double distance=.5pt,dotted] (6,4) arc[x radius=.8,y radius=4,start angle=90,end angle=270];
        \foreach \y in {-2,2} {
          \foreach \x in {0,3,6,9} {
            \begin{scope}[yshift=\y cm,xshift=\x cm]
              \draw[over] (0,.8) ..controls +(.7,0) and +(0,.7).. (1.9,0);
              \draw[over] (3,.8) ..controls +(-.7,0) and +(0,.7).. (1.1,0) ..controls +(0,-.7) and +(-.7,0).. (3,-.8);
              \draw[over] (1.9,0) ..controls +(0,-.7) and +(.7,0).. node[below]{$2$}(0,-.8);
            \end{scope}
          }
        }
        \draw[over,double distance=.5pt] (6,4) arc[x radius=.8,y radius=4,start angle=90,end angle=-90];
        \draw (0,4)--++(12,0) (0,-4)--++(12,0);
        \foreach \x in {3, 9} \foreach \y in {0,1,2} \draw[cap=round,line width=1pt] (\x,-.3+.3*\y)--+(0,0);
      \end{scope}
    \end{tikzpicture}

    \caption{The 3-manifold $\partial_+H$ and its infinite cyclic cover.}
    \label{figure:partial+H-and-cover}
  \end{figure}
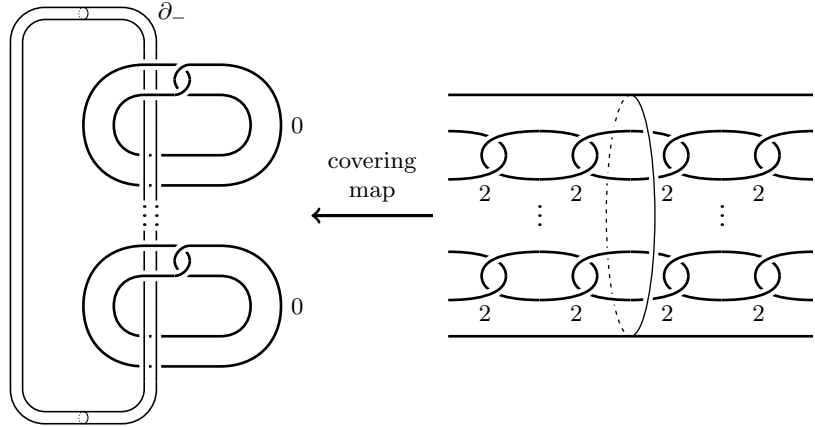

  The left hand side of Figure~\ref{figure:partial+H-and-cover} is a surgery
  description of $\partial(\nu D_0) = \partial_+ \cup \partial_-$ obtained from
  a standard Kirby diagram of the plumbed handle.  More precisely, by choosing
  double point loops for self plumbings, a homeomorphism between $\partial(\nu
  D_0)$ and the 3-manifold given by the surgery description is determined.  For
  the purpose of this section, temporarily choose double point loops whose
  push-offs along the accessory framing are trivial in~$\pi_1(W)=\Z$. This can
  be done by wrapping part of a double point loop on a sheet near the double
  point, around another sheet, if necessary.  (The double point loops used here
  may be different from those in other sections of the article.)  Now, remove the solid torus
  $\partial_-$ and take the infinite cyclic cover of~$\partial_+$.  Note that
  the meridians of the zero-framed circles correspond to push-offs of double
  point loops along the accessory framing, and so they are trivial
  in~$\pi_1(W)=\Z$. It follows that the infinite cyclic cover is given by the
  surgery diagram in the right hand side of
  Figure~\ref{figure:partial+H-and-cover}, which consists of $d$ infinite chains
  lying in $D^2\times \R$.  Observe that the zero framing of the surgery curve
  in the base corresponds to the $\pm 2$ framing of the surgery curve in the
  cover.  The signs of the surgery coefficients in the cover and the signs of
  the clasps are determined by the sign of the double points.

  From the surgery description of the infinite cyclic cover, we obtain
  a presentation of $H_1(\partial_+;\Z[\Z])$ with $d$ generators, say
  $v_i$, and $d$ defining relations $\pm(2-t^{-1}-t)v_i=0$.  It
  follows that $H_1(\partial_+;\Z[\Z]) = \bigoplus^{d}
  \Z[\Z]/(2-t^{-1}-t)$.

  Also, $H_1(\partial X_K;\Z[\Z]) \cong \Z$ is generated by a
  longitude of $K$, which is zero in each of $H_1(\partial_+;\Z[\Z])$
  and $H_1(X_K;\Z[\Z])$.  Therefore, by a Mayer-Vietoris argument for
  $\partial W = \partial_+ \cup X_K$, we obtain the following, from
  which the promised conclusion follows immediately.
  \begin{align*}
    H_1(\partial W;\Z[\Z]) &\cong H_1(X_K;\Z[\Z]) \oplus
    H_1(\partial_+;\Z[\Z])
    \\
    &\cong H_1(X_K;\Z[\Z])\oplus (\Z[\Z]/(2-t^{-1}-t))^{d}. \qedhere
  \end{align*}
\end{proof}

\section{Proof of Alexander polynomial assertions of main theorems}
\label{section:proof-theorem-compute-alex-poly-from-whitney-data}

We begin with a straightforward lemma.

\begin{lemma}\label{lemma:little-UCSS-computation}
  The relative homology $H_2(W,\partial W;\Z[\Z])$ is isomorphic to $\Z[\Z]^d$.
\end{lemma}

\begin{proof}
  We have isomorphisms
  \[
  H_2(W,\partial W;\Z[\Z]) \cong H^2(W;\Z[\Z]) \cong
  \Hom_{\Z[\Z]}(H_2(W;\Z[\Z]),\Z[\Z]) \cong \Z[\Z]^d.
  \]
  The last isomorphism uses that $H_2(W;\Z[\Z]) \cong \Z[\Z]^d$.  The
  second isomorphism uses the universal coefficient spectral
  sequence
  \[
    E^2_{p,q} = \Ext_p^{R}(H_q(W;\Z[\Z]), \Z[\Z]) \Longrightarrow
    H^n(W;\Z[\Z])
  \]
  as we shall now explain.  Since $H_1(W;\Z[\Z]),\Z[\Z])=0$ and $\Z$ has a length
  one projective resolution over $\Z[\Z]$ (see the proof of
  Lemma~\ref{lemma:rank-is-d}), the only surviving $E_2$ term on the
  line $p+q=2$ is $E_2^{0,2}=\Hom_{\Z[\Z]}(H_2(W;\Z[\Z]),\Z[\Z])$.
  The differentials $d_r$ ($r\ge 2$) defined on $E_r^{0,2}$ are
  trivial since $H_1(W;\Z[\Z])=0$ and $\Z[\Z]$ has homological
  dimension two.  Therefore the spectral sequence collapses and we
  have the isomorphism claimed.
\end{proof}

We are ready to connect the pieces of the previous two sections to
prove the Alexander polynomial parts of the main theorems.  The
assertions relating to the Blanchfield form are addressed later in
Section~\ref{section:blanchfield}. Theorem~\ref{theorem:main-blanchfield-alexander}
also uses Lemma~\ref{lemma:spheres-give-basis} below.

\begin{proof}[Proof of Alexander polynomial assertions of
  Theorem~\ref{theorem:compute-alex-poly-from-whitney-data},
  Theorem~\ref{theorem:main-presentation-Blanchfield} and
  Theorem~\ref{theorem:main-blanchfield-alexander}]
  Since $H_1(W;\Z[\Z]) \cong 0$ and $H_2(W;\Z[\Z]) \cong H_2(W,\p
  W;\Z[\Z]) \cong \Z[\Z]^d$, the long exact sequence of a pair yields
  \[
  \Z[\Z]^d \xrightarrow{\Lambda} \Z[\Z]^d \to H_1(\p W;\Z[\Z]) \to 0
  \]
  where $\Lambda$ is the intersection form of $W$.  Since $H_1(\p
  W;\Z[\Z])$ is a torsion module it follows that $\Lambda$ is
  injective.  Indeed
  \begin{align*}
    H_2(\partial W;\Z[\Z]) \cong H^1(\p W;\Z[\Z])
    &  \cong \Ext^1_{\Z[\Z]}(H_0(\p W;\Z[\Z]),\Z[\Z])
    \\
    & \cong \Ext^1_{\Z[\Z]}(\Z,\Z[\Z]) \cong \Z
  \end{align*}
  and any $\Z[\Z]$-module homomorphism from $\Z$ into a free $\Z[\Z]$ module is trivial.

  Represent $\lambda$ as a matrix with respect to the basis for $F
  \subseteq H_2(W;\Z[\Z])$ defined in
  Section~\ref{section:intersection-form}, and with respect to a dual
  basis for $F^* \supseteq \Hom_{\Z[\Z]}(H_2(W;\Z[\Z]),\Z[\Z]) \cong
  H_2(W,\partial W;\Z[\Z])$, so that we obtain a matrix for the
  intersection form of $W$ restricted to $F$.  The presentation for
  $H_1(\p W;\Z[\Z])$ implies that
  \[
  \det(\Lambda) = \ord_{\Z[\Z]}(H_1(\p W;\Z[\Z])) \doteq (t-1)^{2d} \Delta_K(t)
  \]
  up to multiplication by a unit $\pm t^m$.  Here we used
  Proposition~\ref{prop:order-homology-partial-W-new}.  Up to
  multiplication by a unit we have $(t-1)^2 \doteq (1-t)(1-t^{-1}) =
  z$.  For Theorem~\ref{theorem:compute-alex-poly-from-whitney-data},
  the matrix $\Omega$ recording intersection data of the Whitney tower
  satisfies $\lambda = z \Omega$, with $\lambda$ as in
  Theorem~\ref{theorem:intersection-form-using-whitney-discs}.
  Therefore, since $\lambda$ is a $d \times d$ matrix, we have
  \[
  \det(\lambda) = \det (z \Omega) = z^d \det \Omega = (t-1)^{2d}
  \det(\Omega)
  \]
  up to a unit in~$\Z[\Z]$.  Similarly, with $\lambda$ as in
  Theorem~\ref{theorem:intersection-form-using-accessory-discs}, we
  have
  \[
  \det(\lambda) = \det (z \Psi) = z^d \det \Psi = (t-1)^{2d} \det(\Psi)
  \]
  up to a unit in $\Z[\Z]$.

  Now suppose that $F=\pi_2(W)$.  Then $\Lambda = \lambda$ so
  $(t-1)^{2d}\Delta_K(t) = (t-1)^{2d}\det(\Psi)$, and cancelling the
  $(t-1)$ factors yields $\det(\Omega) = \Delta_K(t)$.  Thus Alexander
  polynomial assertion of
  Theorem~\ref{theorem:main-blanchfield-alexander} follows from
  Lemma~\ref{lemma:spheres-give-basis} below.

  In general, we have that $F \subseteq H_2(W;\Z[\Z])$ is a free
  module of the same rank.  We have a commutative diagram:
  \[
  \xymatrix{
    \Z[\Z]^d \ar[r]^{\Lambda} & \Z[\Z]^d
    \ar[d]^{P^*} \\ F \ar[u]^{P} \ar[r]^{\lambda} & F^*}
  \]
  where $P=P(t)$ is represented by a matrix which satisfies $\det(P(1)) = \pm 1$.

  Then we have
  \begin{multline*}
    (t-1)^{2d}\det (\Omega) = \det (\lambda) =
    \det(P(t))\det(\Lambda)\det(P(t)^*) \\
    = \det(P(t))\det(P(t^{-1}))\det(\Lambda) = f(t)f(t^{-1})(t-1)^{2d}
    \Delta_K(t),
  \end{multline*}
  where $f(t):=\det(P(t))$. From this we deduce that, modulo norms
  $f(t)f(t^{-1})$ with $f(1)=\pm 1$, we have $\det(\Omega) =
  \Delta_K(t)$ as claimed.  For
  Theorem~\ref{theorem:main-presentation-Blanchfield}, replace
  $\Omega$ with $\Psi$ in the above argument.  As remarked above,
  Theorem~\ref{theorem:main-blanchfield-alexander} uses
  Lemma~\ref{lemma:spheres-give-basis} below.
\end{proof}

The next lemma completes the proof of the Alexander polynomial
assertions of Theorem~\ref{theorem:main-blanchfield-alexander}, by
showing that in a special case our spheres $S_i$, which generate $F$,
in fact give a basis for $\pi_2(W)$.

Let $D_0 \looparrowright D^4$ be an immersed disc in the $4$-ball with
boundary a knot $K \subset S^3$, where $D_0$ is produced as the track
of a homotopy between $K$ and the unknot, followed by a disc bounding
the unknot, where all self-intersection points of the homotopy occur
at time $1/2$, corresponding to $d$ crossing changes of the knot.
More precisely, let $f \colon S^1 \times I \to S^3$ be a homotopy with
$f(S^1,\{s\})$ isotopic to $K$ for $s <1/2$, isotopic to $U$ for $s
>1/2$, and $f(S^1,\{1/2\})$ a singular knot with $d$ double points.
The track of the homotopy is the image of $g \colon S^1 \times I \to
S^3 \times I$ given by $g(x,s) = (f(x,s),s)$.  Cap off $S^3 \times
\{1\}$ with a copy of $D^4$ and cap off $U \subset S^3 \times \{1\}$
with a standard slice disc for the unknot in this $D^4$.

The Clifford tori for the double points can be surgered into 2-spheres
$S_i$, where $i=1,\dots,d$, using accessory discs, just as in the
construction of the spheres $S_{2i}$ in
Section~\ref{section:construction-of-spheres}. As usual define $W:=
D^4 \sm \nu D_0$.

\begin{lemma}\label{lemma:spheres-give-basis}
The 2-spheres $S_i$ form a basis for $\pi_2(W)$.
\end{lemma}

\begin{proof}
  We construct a handle decomposition for $W$.  Start with a 0-handle
  and a single 1-handle.  Represent this by a Kirby diagram with a
  single dotted unknot.  Perform an isotopy of this unknot until it is
  represented by a diagram having a set of marked crossings
  (potentially a proper subset of all the crossings) which, if
  changed, yield the knot $K$.  At each such crossing, add a single
  $0$-framed 2-handle in the configuration shown in
  Figure~\ref{figure:crossing-change-2-handle}.

  \begin{figure}[t]
    \begin{center}
      \begin{tikzpicture}[scale=.8]
        \node[anchor=south west,inner sep=0,scale=.8] at
        (0,0){\includegraphics[scale=0.12]{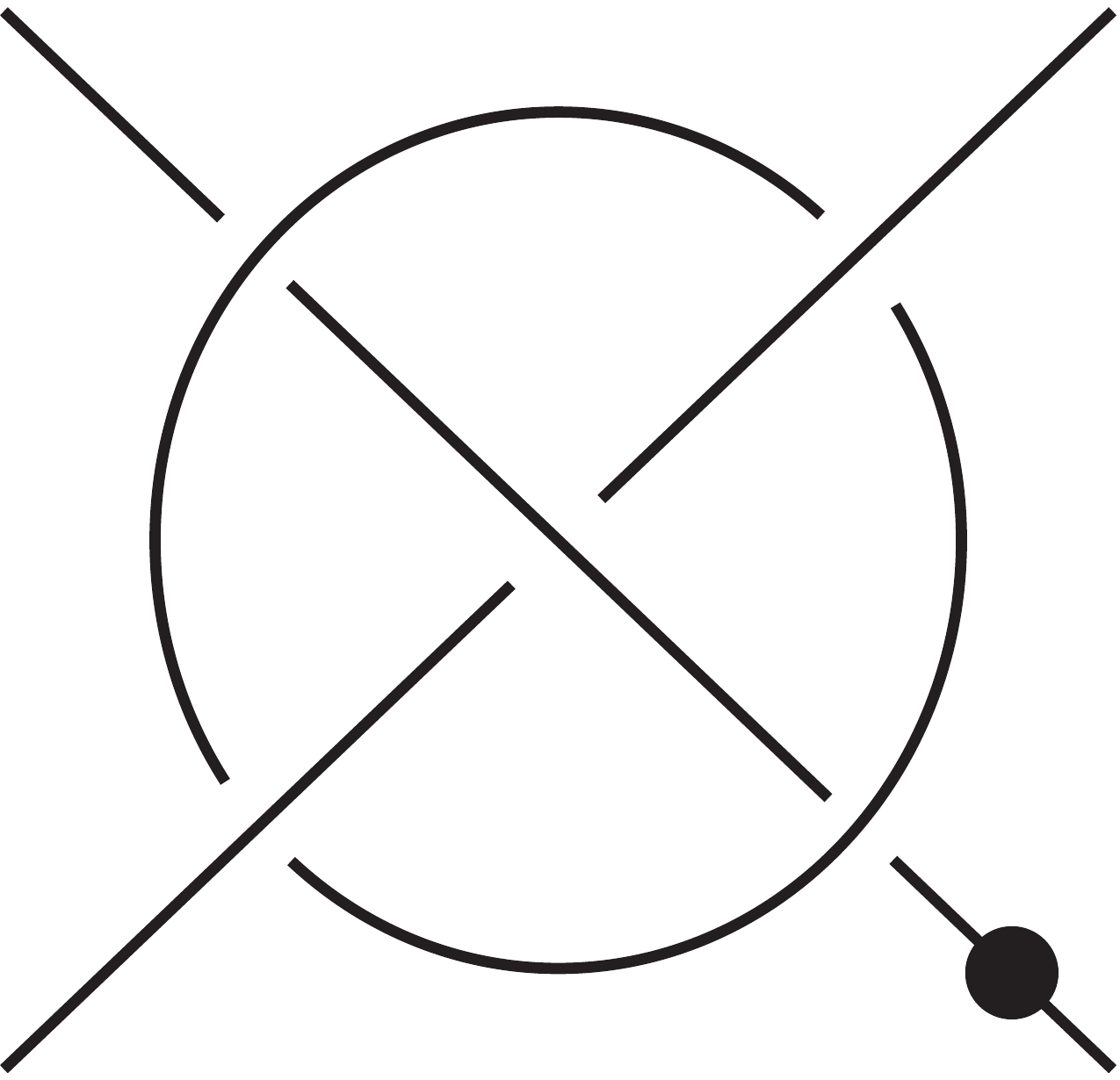}};
        \small
        \node at (.22,1.65)  {$0$};
      \end{tikzpicture}
    \end{center}
    \caption{A 2-handle which we add at each marked crossing of the
      unknotted dotted circle. Altering these crossings
      produces~$K$. This constructs a handle decomposition of the
      exterior of the immersed disc determined by these crossing
      changes.  The two straight strands represent part of the
      unknotted dotted circle.}
    \label{figure:crossing-change-2-handle}
  \end{figure}

  Detailed justification for this can be found in
  \cite[Proposition~3.1]{lightfoot}, which we now summarise.  The
  exterior of an immersed disc constructed by a crossing change on a
  knot can be understood in two steps as follows, which occur at the
  level sets $1/2 \pm \eps$ of the function $F$ given by projection to
  the $I$ factor of $S^3 \times I$, restricted to the exterior of
  $D_0$.  Since we are now passing from $U$ to $K$, we move in the
  direction of decreasing $I$ factor.  First, at $1/2+\eps$, remove a
  small vertical arc which connects the two strands of the crossing.
  One observes that removing the neighbourhood of an arc as described
  does not change the diffeomorphism type i.e.\ the diffeomorphism
  type of $F^{-1}([a,1])$ does not change when $a$ crosses $1/2+\eps$.
  The crossing may be switched by sliding the arcs of the knot (the
  dotted circle) up and down along the removed arc.  Then replace the
  neighbourhood of the vertical arc.  Replacing the arc is equivalent
  to adding the 2-handle as shown in
  Figure~\ref{figure:crossing-change-2-handle}, since this figure
  shows the crossing of the unknotted circle, that is before the
  sliding of the arcs (once the crossing is changed, the 2-handle
  attaching circle bounds a disc in between crossing strands).

  Note that $\pi_1(W) \cong \Z$, since there is a unique 1-handle and
  all 2-handles have no effect on the fundamental group.  A chain
  complex $C_*(W;\Z[\Z])$ is given by (compare
  \cite[Proposition~4.4]{lightfoot})
  \[
  \Z[\Z]^d = C_2 \xrightarrow{0} \Z[\Z] = C_1 \xrightarrow{t-1} \Z[\Z] = C_0.
  \]
  From this we compute $H_2(W;\Z[\Z]) \cong \Z[\Z]^d$ and we note that
  the set of 2-handles give a basis.  The Clifford torus can be seen
  as the core of each 2-handle, union the punctured torus constructed by taking
  a disc bounded by the zero-framed component in
  Figure~\ref{figure:crossing-change-2-handle}, which intersects the
  knot in two points, and tubing along the knot.  The double point
  loop (after suitable twisting) is null homotopic in the complement
  of the standard slice disc for the unknot found in time $s>1/2$,
  therefore the Clifford torus can be surgered to a sphere using the
  procedure of Section~\ref{section:construction-of-spheres}. Since
  the core of the 2-handle is still used precisely once, this
  therefore represents a basis element of $\pi_2(W)$.
\end{proof}

\section{Examples and an algorithm for computation}
\label{section:examples}

\subsection{Using Whitney towers}

We give some examples of computing the Alexander polynomial using
Whitney tower data.  Movies depicting twisted order one Whitney towers
with boundary knots the figure eight knot $4_1$ and the trefoil $3_1$
are shown on the left and right of
Figure~\ref{figure:examples-trefoil-figure-eight} respectively.  The
movies are explained in the caption to the figure.

\begin{figure}[ht]
  \begin{center}
    \begin{tikzpicture}[scale=.7]
      \node[anchor=south west,inner sep=0,scale=.7] at (0,0)
      {\includegraphics[scale=0.5]{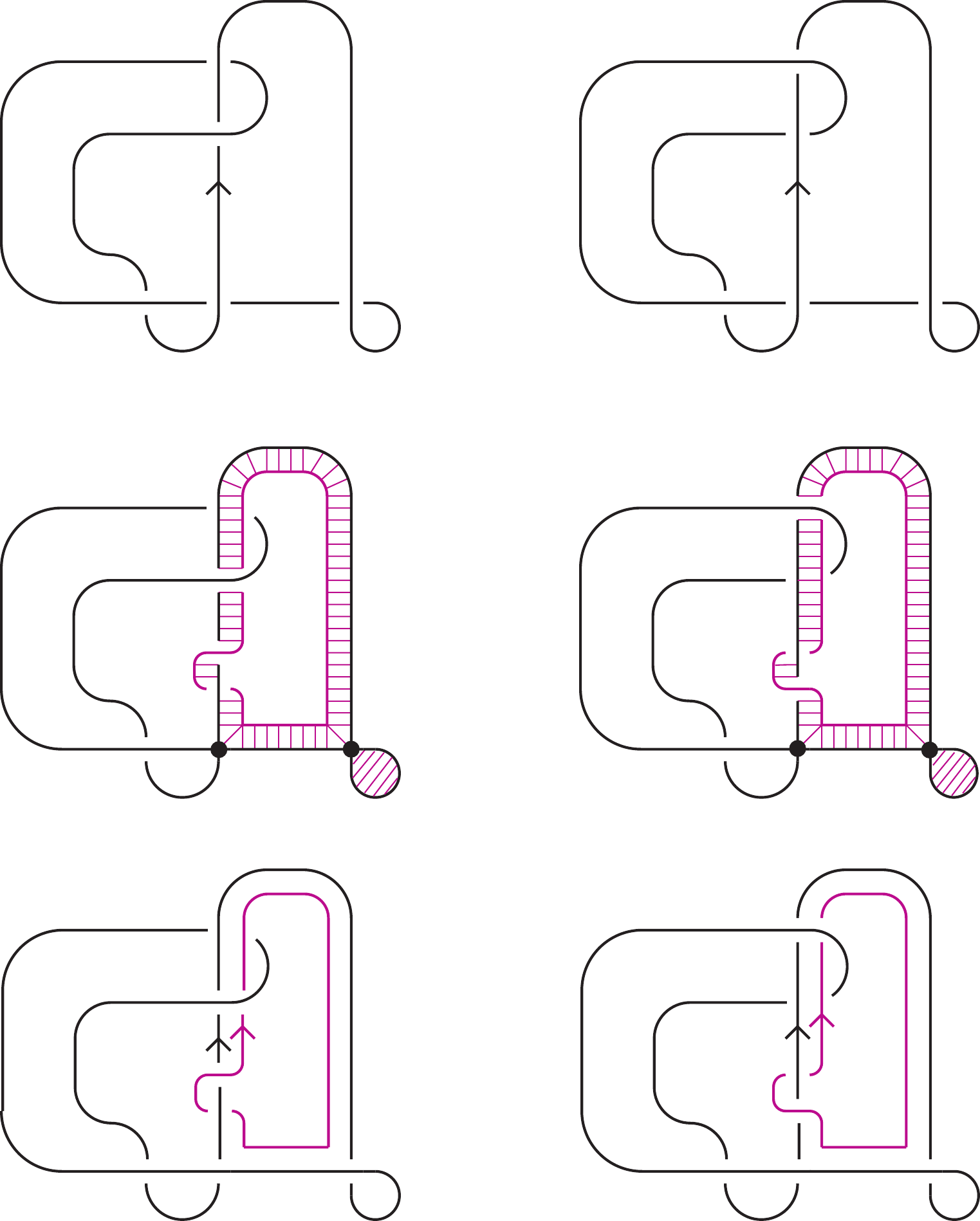}};
      \small
      \node at (3.7,6.8)  {$D_1$};
      \node at (11.4,6.8)  {$D_1$};
      \node at (5.0,6.45)  {$A_1$};
      \node at (12.6,6.45)  {$A_1$};
      \node at (3.75,11.75)  {$4_1$};
      \node at (11.4,11.75)  {$3_1$};
      \node at (3.65,.35)  {unlink};
      \node at (11.3,.35)  {unlink};
    \end{tikzpicture}
  \end{center}
  \caption{Movies for twisted Whitney towers in $S^3 \times I$.
    Movies go top to bottom. Left 3 pictures: Whitney tower cobounding
    the figure eight knot $4_1 \subset S^3 \times \{0\}$. Right 3
    pictures: Whitney tower cobounding the trefoil $3_1 \subset S^3
    \times \{0\}$.  Top picture: the knot, which will evolve via a
    homotopy to the unknot over time, tracing out an immersed
    disc~$D_0$.  Middle picture: double points of the immersion of
    $D_0$, a collar $S^1 \times I$ of the twisted Whitney disc $D_1$,
    and the accessory disc $A_1$.  Bottom picture.  The interior
    boundary of the collar, together with the knot after the crossing
    changes from the double points, form an unlink, which can be
    capped off by two disjoint discs in $S^3 \times \{1\}$, to
    complete $D_0$ and $D_1$. Cap off $S^3 \times I$ with a copy of
    $D^4$ to obtain a Whitney tower in~$D^4$.}
    \label{figure:examples-trefoil-figure-eight}
\end{figure}
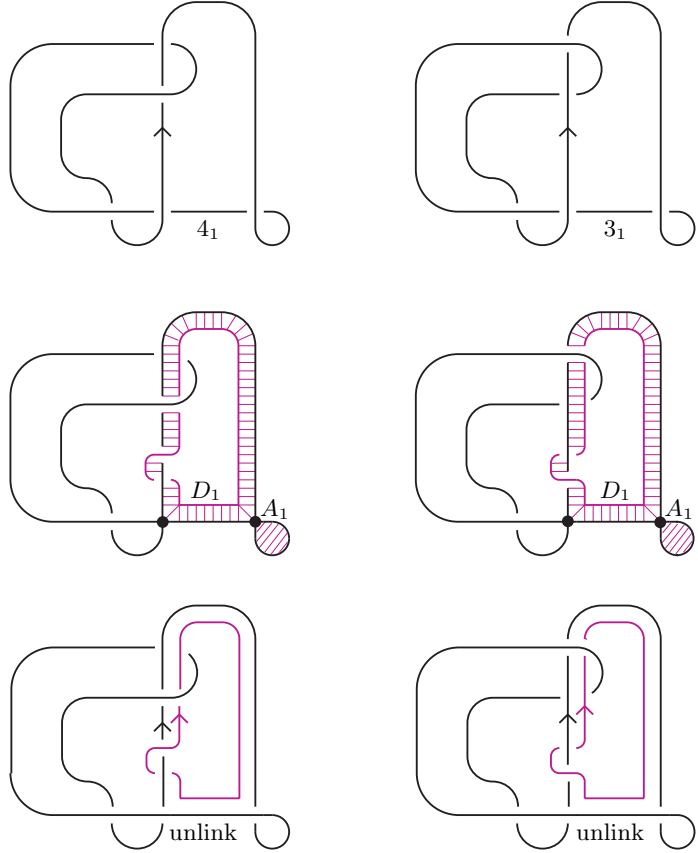

We will simultaneously discuss both examples, indicating differences
between the Whitney towers for $3_1$ and $4_1$ when they arise.  The
only difference turns out to be one sign change.  It is a
straightforward computation to see that $\pi_1(W) \cong \Z$.  Since
there is one Whitney disc and one accessory disc, we have that
$H_2(W;\Z[\Z]) \cong \Z[\Z]^2$, generated by the spheres $S_1$ and
$S_2$, constructed from the Whitney and the accessory disc
respectively, as in Section~\ref{section:construction-of-spheres}
i.e.\ $d=2$.

We apply the formula from
Theorem~\ref{theorem:compute-alex-poly-from-whitney-data}.  The
Whitney and accessory discs are disjointly embedded.  Therefore we
just need to compute the twisting coefficients $a_1$ and $b_1$.  The
accessory disc is untwisted, so $b_1=0$.  The crossing change
occurring during the top-to-bottom evolution of the bottom right of
each knot diagram, where the accessory disc is found in the middle
picture, changes a negative crossing to a positive crossing.  It is
therefore a positive intersection point, so the self-intersection of
$S_2$ is $1$.  On the other hand, the Whitney disc $D_1$ is twisted.
The linking number of the boundary of $D_1$ with the interior of the
collar $S^1 \times I$ in the middle picture, is $+1$ for the figure
eight knot,and $-1$ for the trefoil.  Therefore the twisting of the
Whitney framing relative to the disc framing is $a_1 = -1$ for the
figure eight and $a_1 = +1$ for the trefoil.  This yields the
following intersection matrices $\Omega$, using the formulae given in
the bullet points in Section~\ref{section:intersection-form}.  Recall
that $z= (1-t)(1-t^{-1}) = 2-t-t^{-1}$.  For $4_1$, we have
\[
\Omega = \begin{bmatrix}
  -z & 1 \\ 1& 1
\end{bmatrix},
\]
whose determinant is $-z-1 = t + t^{-1} -2 -1 = t + t^{-1} - 3 \doteq
\Delta_{4_1}(t).$ For $3_1$, we have the matrix
\[
\Omega =
\begin{bmatrix}
  z & 1 \\ 1& 1
\end{bmatrix},
\]
whose determinant is $z-1 = 2 - t - t^{-1} -1 = 1 - t - t^{-1}  \doteq \Delta_{3_1}(t).$

\subsection{An algorithm for computation using accessory discs only}
\label{subsection:algorithm}

By using a natural choice of accessory discs, described below, the
computation of the intersection data (and consequently the abelian
invariants) can be formulated as an algorithm, that we now describe. 

\begin{itemize}
  \item  Fix a given set of crossing changes on a planar diagram of $K$ which convert $K$ to a trivial knot; recall that such a set of crossings can be found on any knot diagram.
  \item Consider the planar diagram obtained by replacing all the
crossings to be changed with a singularity.  This is the diagram at
the level of the intersection points in a movie picture of the immersed disc in $D^4$ arising as the trace of a homotopy realising the crossing changes. The sign of the crossing change determines the sign of the intersection point of~$D_0$.  For each intersection in the singular diagram, draw a double point loop which leaves the crossing, follows along the knot agreeing with the
given orientation, leaving along one strand and returning to the
crossing along the other strand.
\item Push the loop slightly off the
singular knot, and twist the loop around the singular knot until the
linking number with the singular knot is zero i.e.\ after the crossing change the linking number with the resulting unknot is $0$.
\item Choose basing paths
for each accessory loop.
\item Now, replace the singular crossings with the outcome of each of the crossings changes, and apply
an ambient isotopy which takes the resulting unknot to the standard
unknot~$U$.
\item  Under the isotopy, the union of double point loops
becomes an oriented based link, say~$L$.  The $i$th component of $L$
will give rise to the $i$th accessory disc~$A_i$.
\item In the complement
of $U$, apply a homotopy of $L$, that is, crossing changes of $L$,
dragging the basing paths along, until $L\cup U$ becomes the trivial
link.  Here crossing changes involving different components of $L$
are allowed.  For each crossing change on $L$, record the sign of the
crossing change and the element $\ell \in \Z = \pi_1(S^3 \sm \nu U)$
determined by linking with $U$ of the usual concatenation of paths in
$L$ with the basing paths.
\item The number of twists of $L_i$ that we made away from
the blackboard framing, plus twice the signed count of self
intersections of $L_i$, determines the negative of the twisting
of~$A_i$.  With these considerations the intersection data can be completely recovered.
\end{itemize}

\subsection{Examples using accessory discs only}
\label{subsection:examples-accessory-discs-only}

Here is a detailed example of the above algorithm. Consider $K=\Wh_n^-(J)$,
the negatively clasped $n$-twisted Whitehead double of a knot~$J$.
Here negatively clasped means the signs of the crossings are negative,
and $n$-twisted means~$n$ full right handed twists; a negative right
handed twist, which appears if $n<0$, is a left handed twist.  We can
change a single crossing from a negative to a postive crossing in the
clasp to make a homotopy to the unknot $U$.  Therefore $d=1$ and
$\eps_1=1$.  The double point loop becomes a copy of the knot $J$,
twisted $-n$ times around this unknot.  Add $n$ twists to the double
point loop so that it is null homotopic in the complement of $U$.  The
null homotopy of~$J$ produces the accessory disc $A_1$.  Every double
point of $A_1$ has the trivial element of
$\pi_1(S^3 \sm \nu U) \cong \Z$ associated to it. Add local cusps of
the appropriate sign so that the signed count of double points of
$A_1$ vanishes.  The matrix $\Psi$ is then a $1 \times 1$ matrix with
entry $1 + zb_1$, where $b_1$ is the twisting coefficient.  Since we
added $n$ positive twists to the double point loop, the twisting
coefficient is $-n$, and we compute:
\[
\Delta_{K}(t) \doteq \det(\Psi) = 1-nz = 1-n(2-t-t^{-1}) = 1-2n + nt + nt^{-1}.
\]

\section{The Seifert form and the Arf invariant}\label{section:defns-of-Arf}

We are about to investigate the implications of
Theorem~\ref{theorem:compute-alex-poly-from-whitney-data} for the Arf
invariant of a knot.  First, in this section, we briefly recall the
usual definition of the Arf invariant of a knot in terms of a Seifert
form.

\begin{definition}\label{defn:arf-invariant-quadratic-form}
  A \emph{quadratic enhancement} of a symmetric bilinear form $\lambda
  \colon M \times M \to \Z_2$ on a $\Z_2$ vector space $M$ is a
  function $q \colon M \to \Z_2$ such that
  \[
  q(x) + q(y) + q(x+y) \equiv \lambda(x,y) \mod{2}	
  \]
  for all $x,y \in M$.  A quadratic form is a symmetric bilinear form
  $M,\lambda$ together with a quadratic enhancement $q$.

  Let $\{e_1,f_1,e_2,f_2,\dots,e_n,f_n\}$ be a symplectic basis for
  $M$ i.e.\ $\lambda(e_i, e_j) = 0$, $\lambda(f_i ,f_j) =0$ and
  $\lambda(e_i , f_j) = \delta_{ij}$ for all $i,j = 1,\dots,n$.  Then
  the Arf invariant of the quadratic form is
  \[
  \Arf (M,\lambda, q) := \sum_{i=1}^n\, q(e_i) q(f_i) \mod{2}.
  \]
  See \cite[Appendix]{Rourke-Sullivan-71} for the proof that this is
  well-defined.

\end{definition}

\begin{definition}\label{defn:arf-invariant-seifert-surface}
  We will define a $\Z_2$-valued quadratic enhancement on the
  $\Z_2$-valued intersection form on the first homology of a Seifert
  surface $F$ of the knot.  Represent an element of $H_1(F;\Z_2)$ by
  an oriented simple closed curve $\gamma \subset F$, and define a
  framing of its normal bundle by choosing a framing of the normal
  bundle $\nu_{F \subset S^3}$ of $F$ in $S^3$. (Using the orientation
  of $S^3$ and $\gamma$ this choice determines a framing of the normal
  bundle $\nu_{\gamma \subset F}$, and therefore a framing of
  $\nu_{\gamma \subset S^3}$ in the conventional sense).  Every simple
  closed curve in $S^3$ bounds a closed oriented Seifert surface
  $G_{\gamma}$, and the unique (up to homotopy) framing of
  $\nu_{\gamma \subset S^3}$ which extends to a framing of the normal
  bundle of $G_{\gamma}$ is the \emph{zero framing} of $\gamma$.  We
  may therefore compare the zero framing of $\gamma$ with the framing
  defined above by the embedding of $F$, to obtain an integer.  This
  measures the number of full twists in the ``band'' of the Seifert
  surface with core $\gamma$.  The modulo 2 reduction of this integer
  defines a quadratic enhancement of the $\Z_2$-intersection form on
  $H_1(F;\Z_2)$, as promised, that is a function $q \colon H_1(F;\Z_2)
  \to \Z_2$.  The Arf invariant of $K$ is
  $\Arf(H_1(F;\Z_2),\lambda,q)$.
\end{definition}

\section{Proof of Arf Invariant Theorem \ref{theorem:arf-whitney-alexander-poly}}
\label{section:proof-Arf-invariant-theorem}

\begin{proof}[Proof of Theorem~\ref{theorem:arf-whitney-alexander-poly}]
  We saw in the proof of
  Theorem~\ref{theorem:compute-alex-poly-from-whitney-data} that
  $\det(\Omega(t)) = \Delta_K(t) f(t)f(t^{-1})$ for some $f \in
  \Z[t,t^{-1}]$ with $f(1)= \pm 1$.  Thus $\det(\Omega(-1)) =
  \Delta_K(-1)f(-1)^2$.  But $f(1)=\pm 1$ implies that $f(-1)$ is odd.
  Thus $f(-1)^2 \equiv \pm 1 \mod{8}$, and so we have that
  $\Delta_K(-1) \equiv \det(\Omega(-1)) \mod{8}$.  Then observe that
  $\Omega(t)=\lambda(t)/z$, so $\Omega(-1) = \lambda(-1)/4$.  The form
  of $\lambda$ in
  Theorem~\ref{theorem:intersection-form-using-whitney-discs} implies
  that $\Omega(-1)$ reduces to the matrix $A$ in
  Lemma~\ref{lemma:levine-argument-modified} below, with $X(-1)=B$,
  $Y(-1)=C$, $x_i=4a_i$ and $y_i=4b_i$.

 \begin{lemma}\label{lemma:levine-argument-modified}
   Let $A$ be a $d \times d$ matrix over $\Z$, with $d=2k$, of the
   form $B + 4C + 4C^T$ where $C$ is upper triangular and $B$ is a
   block diagonal sum of $2 \times 2$ matrices $D_i$ of the form
   $\sbmatrix{x_i & 1 \\ 1 & 1+ y_i}$, where $x_i$ and $y_i$ are both
   a multiple of $4$.  Then $\det A \equiv (-1)^k +\sum_{i=1}^k x_i
   \mod{8}$.
\end{lemma}

In particular,
\[
(-1)^k + \sum_{i=1}^k 4a_i \equiv
\begin{cases}
  \pm 1 & \text{if } \sum_{i=1}^k a_i \equiv 0 \mod{2} \\
  \pm 3 & \text{if } \sum_{i=1}^k a_i \equiv 1 \mod{2}.
\end{cases}
\]
The count on the right hand side is exactly the number of twisted
Whitney discs modulo two.  This completes the proof of
Theorem~\ref{theorem:arf-whitney-alexander-poly} modulo the proof of
Lemma~\ref{lemma:levine-argument-modified}.
\end{proof}

The idea for Lemma~\ref{lemma:levine-argument-modified} and its proof
come from \cite[Section~3.5]{Levine:1966-1}.  The argument in this
lemma is slightly simpler since the contributions from the accessory
discs are always odd, thus the Whitney disc terms decide the outcome
modulo 8.  In the Seifert surface case considered by Levine, the
twisting of both of a dual pair of generators determine whether that
dual pair contributes to the Arf invariant.

\begin{proof}[Proof of Lemma~\ref{lemma:levine-argument-modified}]

  Following Levine, we call an element of the matrix~$A$
  \emph{special} if it is odd: these are the entries $a_{(2i-1),(2i)}
  = a_{(2i),(2i-1)}$ and $a_{(2i),(2i)}$, for $i=1,\dots,k$.  The
  remaining entries of $A$ are even and these are called
  \emph{non-special}.

  The determinant is computed as a sum of terms, where each term is a
  product of elements, one taken from each row and each column.  Note
  that all the non-special terms are in fact a multiple of $4$.  Thus
  in order for a summand of the determinant to contribute to the
  reduction modulo 8 it must be a product of elements, at most one of
  which is non-special.

  We therefore need to look at the summand containing only special
  terms (there is precisely one such summand) and the summands
  containing precisely one non-special term.  The only summand of the
  determinant which contains only special terms is
  \[
  \prod_{i=1}^k - a_{(2i-1),(2i)} a_{(2i),(2i-1)} = (-1)^k \prod_{i=1}^k a_{(2i-1),(2i)}^2 .
  \]
  Since $a_{(2i-1),(2i)} = 1+ 4n_i$ for some $n_i \in \Z$, we have
  that $a_{(2i-1),(2i)}^2 \equiv 1 \mod{8}$ so that modulo $8$ the
  contribution is $(-1)^k$.

  There is a summand with precisely one non-special term for each
  $i=1,\dots,k$, of the form:
  \begin{multline*}
    a_{(2i-1),(2i-1)}a_{(2i),(2i)} \prod_{1 \leq j < i, i < j \leq k} - a_{(2j-1),(2j)} a_{(2j),(2j-1)} \\
    = (-1)^{k-1} a_{(2i-1),(2i-1)}a_{(2i),(2i)} \prod_{1 \leq j < i, i < j \leq k}  a_{(2j-1),(2j)}^2.
  \end{multline*}

  Let $\a_i$ and $\beta_i$ be such that $x_i=4\a_i$ and
  $y_i=4\beta_i$.  We also may write $a_{(2i-1),(2i-1)} = 4 \a_i + 8
  m_i$, $a_{(2i),(2i)} = 1+ 4 \beta_i + 8 \ell_i$ and finally
  $a_{(2j-1),(2j)} = 1+ 4n_j$ as above.  Thus modulo 8 we have that
  each summand
  \[
  (-1)^{k-1} a_{(2i-1),(2i-1)}a_{(2i),(2i)} \prod_{1 \leq j < i, i < j
    \leq k} a_{(2j-1),(2j)}^2 \equiv 4 \a_i \mod{8}.
  \]

  Combining the contributions to the determinant of the summand with
  all special terms and the $k$~summands with precisely one
  non-special term, we have that $\det A = (-1)^k + \sum_{i=1}^k x_i$
  as claimed.
\end{proof}

\section{The Blanchfield form}\label{section:blanchfield}

  In this section we show that the matrices $\Omega$ and
$\Psi$ present a linking form in the Witt class of the Blanchfield
form of $K$, and that the form they present is isometric to the
Blanchfield form of $K$ in the case that the immersed disc $D_0$
arises from crossing changes on $K$.  This will prove the Blanchfield
form statements of
Theorems~\ref{theorem:compute-alex-poly-from-whitney-data},
\ref{theorem:main-presentation-Blanchfield}
and~\ref{theorem:main-blanchfield-alexander}.

Let $R=\Z[\Z]$, and let $Q=\Q(\Z)$ be its quotient field.  A
\emph{linking form} is defined to be a sesquilinear, hermitian,
nonsingular form $\beta\colon V\times V \to Q/R$ with $V$ a finitely
generated torsion $R$-module.  Suppose $M$ is a 3-manifold over $\Z$,
that is, $M$ is endowed with a homomorphism $\pi_1(M)\to \Z$.  Suppose
$H_1(M;R)$ is torsion over $R$ and the map $H_1(\partial M;R) \to H_1(M;R)$ is the zero map.
  Then the \emph{Blanchfield form} \cite{Blanchfield:1957-1} of $M$ is defined to
be the linking form
\[
\Bl\colon H_1(M;R) \times H_1(M;R) \to Q/R,
\]
whose adjoint $\Bl^*$ coincides with the composition of isomorphisms
\begin{multline*}
  H_1(M;R) \to H_1(M,\partial M;R) \to H^2(M;R) \\
  \to H^1(M;Q/R) \to
  \ol{\Hom_{R}(H_1(M;R),Q/R)}.
\end{multline*}
That is, $\Bl^*(y)(x) = \Bl(x,y)$.  Here the bar denotes the use of
the involution on $R$ to convert from a right module to a left module.
The morphisms above are given by the long exact sequence of the pair
$(M,\partial M)$, Poincar\'e duality, a Bockstein homomorphism and
universal coefficients, respectively.  The proof that they are all
isomorphisms can be found in~\cite{Levine-77-knot-modules}.

The Blanchfield form of an oriented knot $K$ is defined to be that of
the exterior $X_K$ endowed with the homomorphism $\pi_1(M)\to \Z$ that sends a
positive meridian to the generator~$t$.

\begin{definition}
  \label{defn:presentation-of-blanchfield-form}
  We say that an $n\times n$ hermitian matrix $A=A(t)$ over $R$
  \emph{presents} a linking form $\beta$ if $\beta$ is isometric to
  the sesquilinear pairing
  \[
  (R^{n}/A\cdot R^n) \times (R^n/A\cdot R^n) \to Q / R
  \]
  given by $([x],[y]) \mapsto - y^T \cdot \overline{A}^{-1} \cdot \overline{x}$.
\end{definition}

\begin{lemma}
  \label{lemma:boundary-blanchfield}
  Suppose $W$ is a 4-manifold with $\pi_1(W)=\Z$ so that
  $\partial W=M$ is over $\Z$, and $\Lambda=\Lambda(t)$ is a matrix
  representing the $R$-valued intersection form on $H_2(W;R)$.  Then
  $\Lambda$ presents the Blanchfield form of~$M$.
\end{lemma}

\begin{proof}
  First, $H_1(W;R)=0$ since $\pi_1(W)=\Z$.  In addition, $H_2(W;R)$
  and $H_2(W,M;R)$ are free $R$-modules of the same rank, by
  Lemmas~\ref{lemma:pi-2-free-module}
  and~\ref{lemma:little-UCSS-computation}.  Since $\Lambda$ represents
  $H_2(W;R) \to H_2(W,M;R)$, it follows that $\Lambda$ is a
  presentation matrix for $H_1(M;R)$.  Here we fix an arbitrary basis
  for $H_2(W;R)$ and use the dual basis for $H_2(W,M;R)$ as usual.

  Let $\Phi$ be the composition
  \begin{multline*}
    \Phi\colon H_2(W,M;R) \to H_2(W,M;Q) \xrightarrow{\cong}
    H_2(W;Q) \to H_2(W;Q/R)
    \\
    \xrightarrow[PD]{\cong} H^2(W,M;Q/R) \to
    \overline{\Hom_R(H_2(W,M;R),Q/R)}.
  \end{multline*}
  It is known that $\Bl(\partial x, \partial y) = -\Phi(x)(y)$ for any
  $x,y\in H_2(W,M;R)$; see for instance \cite[Lemma~3.3]{Cha:2006-1}.
  (Our sign convention is opposite to that of~\cite{Cha:2006-1}.)

  Using bases for $H_2(W;-)$ induced by our fixed basis and using the
  dual basis for $H_2(W,M;-)$ once again, all the arrows but the
  second in the above definition of $\Phi$ are represented by the
  identity matrix.  The second arrow is the inverse of the
  intersection pairing, and thus represented by~$\Lambda^{-1}$.  It
  follows that
  \[
    \Bl(\partial x,\partial y) = - y^T \cdot \overline{\Lambda^{-1}x} =
    - y^T\cdot\overline{\Lambda}^{-1}\cdot\overline{x}. \qedhere
  \]
\end{proof}

Now consider the case of $W=D^4\sm \nu D_0$.  Construct the following
commutative diagram, as explained below the diagram:
\[
  \xymatrix@R=2em@C=3em{
    H_2(W;R) \ar[r]^-{\Lambda} \ar@{=}[dd] &
    H_2(W,\partial W;R) \ar[r]^-{\partial} \ar@{=}[d]^{\text{Lemma~\ref{lemma:little-UCSS-computation}}} &
    H_1(\partial W;R) \ar@{=}[d]^{\text{Proposition~\ref{prop:order-homology-partial-W-new}}} \ar[r] &
    0
    \\
    &
    H_2(W;R)^* &
    \hbox to 0mm{\hss$H_1(X_K;R)\oplus (R/\langle z\rangle)^d$\hss}
    \\
    H_2(W;R) \ar[r]^-{A} &
    N \ar@{_(->}[]+<0mm,+2.5ex>;[u]_-{zI} \ar[r]^-{\partial|_N} &
    H_1(X_K;R) \ar@{_(->}[]+<0mm,+2.5ex>;[u]_-{\text{summand}} \ar[r] &
    0
  }.
\]
The top row is a part of the long exact sequence for $(W,\partial W)$.
Let $N:=z\cdot H_2(W;R)^*$.  Since $H_2(W;R)^* \cong R^d$,
$N\cong R^d$.  Since
$\partial(N) \subset z\cdot H_1(\partial W;R) \subset H_1(X_K;R)$,
$\partial$ induces $\partial|_N \colon N \to H_1(X_K;R)$.  Since $1-t$
is an automorphism on $H_1(X_K;R)$, so is $z=(1-t)(1-t^{-1})$.  It
follows that $\partial|_N$ is surjective.  Also, since
$H_2(W;R)^*/N\cong (R/\langle z\rangle)^d$, the image of $\Lambda$
lies in~$N$.  So there is $A\colon H_2(W;R) \to N$ making the diagram
commute.  Multiply our (dual) basis for $H_2(W;R)^*$ by $z$ to obtain
a basis for~$N$.  With respect to this, the inclusion
$N\to H_2(W;R)^*$ is (represented by) the diagonal matrix $zI$.  So
$\Lambda=z A$ as matrices.

\begin{claim*}
  The matrix $A$ presents the Blanchfield form $\Bl_{X_K}$ of~$K$.
\end{claim*}

To prove the claim, first observe that the Blanchfield form
$\Bl_{\partial W}$ of $\partial W$ is given by
\[
  \Bl_{\partial W}(\partial u,\partial v) = -y^T \overline\Lambda^{-1}
  \overline{x} = -z^{-1} v^T \overline A^{-1} \overline{u}
\]
for $u,v\in H_2(W;R)^*$, by Lemma~\ref{lemma:boundary-blanchfield}.
Using that the bottom row of the above diagram is exact, identify
$H_1(X_K;R)$ with $R^d/A\cdot R^d = N/\Im\{A\}$.  Then, from the above
description of $\Bl_{\partial W}$, it follows that
\[
  \Bl_{X_K} \colon (R^d/A\cdot R^d) \times (R^d/A\cdot R^d) \to Q/R
\]
is given by
$(x,y) \mapsto \Bl_{\partial W} (x, y) = -y^T (z\overline A^{-1})
\overline x$.
Since $z=(1-t)(1-t^{-1})$, it is straightforward to see that the
following diagram is commutative:
\[
  {\newsavebox{\tempbox}\sbox{\tempbox}{$(\R^d/A\cdot R^d) \times {}$}
    \xymatrix@C=6em{
      (R^d/A\cdot R^d) \times (R^d/A\cdot R^d)
      \ar@<-.5\wd\tempbox>[d]_{1-t} \ar@<.5\wd\tempbox>[d]^{1-t} \ar[r]^-{-z\overline A^{-1}} &
      Q/R \ar@{=}[d]
      \\
      (R^d/A\cdot R^d) \times (R^d/A\cdot R^d) \ar[r]_-{-\overline{A}^{-1}} &
      Q/R
    }
  }
\]
Since $1-t$ is an automorphism on $R^d/A\cdot R^d=H_1(X_K;R)$, it
follows that $A$ presents~$\Bl_{X_K}$, as claimed above.

In the case that the submodule $F$ generated by our 2-spheres is equal
to $H_2(W;R)$, for example in the special case of an immersed disc
arising from crossing changes on $K$, we have $\Lambda=\lambda=z\Psi$,
that is, $A=\Psi$.  This completes the proof of the Blanchfield form
assertion of Theorem~\ref{theorem:main-blanchfield-alexander}.

In general, namely when $F$ is not necessarily $H_2(W;R)$, let
$P=P(t)$ be the square matrix representing the inclusion
$R^d \cong F\to H_2(W;R)\cong R^d$.  The matrix $P(1)$ is unimodular
over $\Z$, since our spherical basis elements of $F$ descend to a
basis of $H_2(W;\Z)$.

Construct the following commutative diagram, as explained below:
\[
  \xymatrix{
    H_2(W;R) \ar[r]^-A &
    N \ar@{^(->}[]+<1.1em,0mm>;[r]^-{zI} \ar[d]^{P^*}&
    H_2(W;R)^* \ar[d]^-{P^*}
    \\
    F \ar[r]_{\Omega} \ar[u]^-{P} &
    zF^* \ar@{^(->}[]+<1.3em,0mm>;[r]_-{zI}& F^* }
\]
First, choosing the natural basis for $zF^*\subset F^*$ as we did for
$N\subset H_2(W;R)^*$, the inclusion $zF^* \hookrightarrow F^*$ is the
diagonal matrix $zI$.  Since $F\to F^*$ is the intersection matrix
$\lambda=z\Omega$ (or~$z\Psi$), the map $F\to zF^*$ is represented by
the matrix $\Omega$ as in the above diagram.  Since $P^*$ is
$R$-linear, it takes $N=z\cdot H_2(W;R)^*$ to~$zF^*$, namely $P^*$
induces the middle vertical arrow in the above diagram.  Furthermore,
with respect to our basis for $zF^*$, the induced homomorphism
$N\to zF^*$ is represented by the same matrix~$P^*$.

From the above diagram, it follows that $\Omega=PAP^*$.  By the
following lemma, $\Omega$ presents a linking form which is Witt
equivalent to the Blanchfield form of~$X_K$.  This completes the proof
of the Blanchfield form assertions of
Theorems~\ref{theorem:compute-alex-poly-from-whitney-data}
and~\ref{theorem:main-presentation-Blanchfield}.

\begin{lemma}[Ranicki]
  The two linking forms presented by hermitian matrices $A(t)$ and
  $P(t)A(t)P(t^{-1})^T$ are Witt equivalent, where $\det P(1)=\pm1$
  and $\det A(t) \neq 0$.
\end{lemma}

\begin{proof}
  This lemma appears on Ranicki \cite[p.~268]{Ranicki:1981-1}, in the
  proof of his Proposition 3.4.6~(ii).  To make the translation from
  Ranicki's notation to ours without having to read too much of
  \cite{Ranicki:1981-1}, one needs to know that the boundary of a form
  is the linking form presented by a matrix representing that form,
  and the fact that $\det P(1)=\pm1$ implies that $P$ is an
  isomorphism over $\Q(\Z)$, that is $P$ corresponds to an
  $S$-isomorphism, with $S$ the nonzero polynomials in $\Z[\Z]$.
\end{proof}

\bibliographystyle{alpha}
\def\MR#1{}
\bibliography{research}

\end{document}